\documentclass{elsarticle}

\usepackage[english]{babel}
\usepackage{amsmath,amsfonts}
\usepackage{epsf}
\usepackage{psfrag}
\usepackage{epsfig}
\usepackage{subfigure}
\usepackage{color}
\graphicspath{{}}

\date{\today}
\def\eps{\varepsilon}
\newcommand{\epsp}{\varepsilon_{\mathrm{p}}}
\def\R{{\mathbb R}}
\def\C{\mathbb C}
\def\Z{{\mathbb Z}}
\def\N{{\mathbb N}}

\newcommand{\lbdspace}{H^{\ltrace}_{\gamma}}

\def\rank{\,\mathrm{rank}\,}
\def\diag{\,\mathrm{diag}\,}

\newcommand{\xiup}{\xi^{\mathrm{r}}} 
\newcommand{\abs}[1]{\left| #1 \right|}
\newcommand{\Norm}[1]{\left\| #1 \right\|}
\newcommand{\paren}[1]{\left( #1 \right)}
\newcommand{\braces}[1]{\left\{ #1 \right\}}
\newcommand{\bracket}[1]{\left[ #1 \right]}
\newcommand{\lsp}{\langle}
\newcommand{\rsp}{\rangle}

\newcommand{\diffq}[2]{\frac{\partial #1}{\partial #2}}
\newcommand{\calL}{\mathcal{L}}
\newcommand{\calS}{\mathcal{S}}

\newcommand{\zerotwo}{(0,0)^{\top}}
\newcommand{\Scomp}{\tilde{\mathcal{S}}}
\newcommand{\ltrace}{\tau}
\newcommand{\vp}{v^{+}}
\newcommand{\vr}{v^{\rm r}}
\newcommand{\Tref}{T_{\rm r}}
\newcommand{\mind}{m}
\newcommand{\nind}{n}

\usepackage{amsthm}
\newtheorem{theorem}{Theorem}[section]
\newtheorem{lemma}[theorem]{Lemma}
\newtheorem{proposition}[theorem]{Proposition}
\newtheorem{corollary}[theorem]{Corollary}

\newdefinition{definition}[theorem]{Definition}
\newdefinition{remark}[theorem]{Remark}
\newdefinition{assumption}[theorem]{Assumption}

\newcommand{\beq}{\begin{equation}}
\newcommand{\eeq}{\end{equation}}

\begin{document}
\begin{frontmatter}

\title{Riesz bases and Jordan form of the translation operator in semi-infinite
periodic waveguides}

\author[nam]{T.~Hohage}
\ead{hohage@math.uni-goettingen.de}
\author[ul]{S.~Soussi\corref{cor1}\fnref{DFG}}
\ead{sofiane.soussi@gmail.com}

\fntext[DFG]{Most of this research was carried out while the second author was working in G{\"o}ttingen supported by the DFG grant RTG 1023.}
\cortext[cor1]{Corresponding author}

\address[ul]{Department of Mathematics and Statistics, University
of Limerick, Limerick, Ireland.}
\address[nam]{Institut f{\"u}r Numerische und Angewandte Mathematik,
University of G{\"o}ttingen, G{\"o}ttingen,
Germany.}

\begin{abstract}
We study the propagation of time-harmonic acoustic or transverse
magnetic (TM) polarized electromagnetic waves in a periodic waveguide
lying in the half-strip $(0,\infty)\times(0,L)$. 
It is shown that there exists a Riesz basis
of the space of solutions to the time-harmonic wave equation such that
the translation operator shifting a function by one periodicity length to the left
is represented by an infinite Jordan matrix which contains at most
a finite number of Jordan blocks of size $> 1$. Moreover, 
the Dirichlet-, Neumann- and mixed traces of this Riesz basis on the left
boundary also form a Riesz basis. Both the cases of frequencies in a band gap 
and frequencies in the spectrum and a variety of boundary conditions on the
top and bottom are considered.  
\end{abstract}

\begin{keyword}
photonic crystal \sep photonic band gap \sep periodic dielectric
medium \sep Floquet theory \sep analytic theory of operators
\MSC 35P30 \sep 78A40 \sep 35Q60  \sep 35B27  \sep 35B30
\end{keyword}


\end{frontmatter}

\section{Introduction}
Periodic media have received much attention in recent years since 
they can prohibit the propagation of electromagnetic
and acoustic waves in some frequency ranges \cite{yabl}.
This localization property is a consequence of band structure of
the spectrum of the underlying differential operator and of the
presence of band gaps in this spectrum. Waves with frequencies in a band gap will 
decrease exponentially inside such media as a consequence of the 
exponential decay of the Green's kernel, see \cite{FigoKlein1}.
The band structure of the spectrum is explained by Floquet theory, 
see \cite{FigoKuch1,FigoKuch2,Kuchment_Floquet}. We 
refer the reader to \cite{Kuchment} for a mathematical introduction to
photonic crystals and to  \cite{joa2,Sakoda} for a physical introduction.

The strong localization property can be used to construct devices
that mould the flow of light or sound at a very small length scale.
The simulation of such devices requires the numerical solution of
differential equations in locally perturbed periodic media, 
which is a challenging task. The proof of existence and numerical
computation of defect modes has been studied in 
\cite{AS,FigoKlein4,FigoKlein5,FigoKlein1,FigoKlein3,FigoKlein2,soussi3}. 
Problems in locally perturbed infinite periodic media 
with a source term considered as part of the problem have been studied
more recently in \cite{FlissJoly:08,Fliss:09,JolyFlissLi}.
Here a natural approach consists in solving a boundary value problem on
a compact set enclosing the perturbation and using a Dirichlet-to-Neumann
or a related operator on the artificial boundary.
%
Hence, the problem is decomposed
into an interior and an exterior problem. The purpose of this paper
is to contribute to the understanding of exterior boundary value problems
in semi-infinite periodic waveguides. A summary of our results has already
been given in the abstract, for a precise formulation we refer to the next
section.

The study of wave propagation in doubly periodic half planes  
can be reduced via the Floquet transform to wave propagation in
semi-infinite periodic waveguides with quasi-periodic boundary conditions 
on the lateral boundaries. Moreover, using a clever trick proposed by
Fliss \& Joly
\cite{FlissJoly:08}, boundary value problems in the doubly periodic exterior 
of a square can be reduced in some sense to boundary value problems
in a doubly-periodic half plane. Therefore, the analysis of this paper
is also relevant for these problems.


Our results prove two open conjectures formulated in connection with
a numerical method to compute the Neumann-to-Dirichlet map proposed by
Fliss, Joly and Li (see the discussion of Theorem \ref{theo:main}). 
Actually, this study arose from attempts to justify an alternative numerical approach,
which will be published elsewhere. Moreover, our results
explain with a new approach the exponential decay of waves in
periodic media with frequency in the band gap studied in
\cite{FigoKlein4,FigoKlein5,FigoKlein1,FigoKlein3,FigoKlein2}
and even provides the
optimal decay rate which corresponds to the one of the slowest decaying
Floquet mode, which may provide guidance for photonic
crystal optimization by focusing on this Floquet mode and
trying to make its decay as fast as possible.

The paper is outlined as follows: in {\S}
\ref{sec:the_problem} we introduce the problem and present the main theorem.
Some known prerequisites are collected in  {\S} \ref{sec:preliminaries},
in particular a generalization of Rouch{\'e}'s theorem 
shown in \cite{GohbergSigal}, which will serve as an essential tool in the following analysis.
The remaining part of the paper is dedicated to the proof of the main theorem
and will be summarized at the end of section \ref{sec:preliminaries}.
There are two appendices on radiation conditions and on uniqueness results.


\section{Statement of the problem and the main results}
\label{sec:the_problem}
The propagation of time-harmonic acoustic or transverse magnetic (TM) 
polarized electromagnetic waves in a 2-D waveguide lying in the half-strip
$S^+:=\R^+\times(0,L)$ is described by the differential equation
\begin{subequations}\label{dirprob}
\begin{equation}\label{eq:pde}
\Delta v + \omega^2 \epsp v =0, \qquad \text{in }S^+.
\end{equation}
We assume that $\epsp\in L^{\infty}(S)$ with $S:=\R\times(0,L)$ 
is periodic with period length $1$ in the first variable and bounded away from $0$, i.e.
\begin{alignat*}{2}
&\epsp(x_1+1,x_2)=\epsp(x_1,x_2),\qquad
&&\mbox{for all } (x_1,x_2)\in S,\\
&0<\mathrm{ess inf}\epsp\leq \epsp\leq \bar\varepsilon\qquad 
&&\mbox{a.e.\ in } S
\end{alignat*}
\begin{figure}[ht]
\centering
\begin{psfrags}
\psfrag{zero}[][][1]{$0$}
\psfrag{one}[][][1]{$1$}
\psfrag{L}[][][1]{$L$}
\psfrag{x1}[][][1]{$x_1$}
\psfrag{x2}[][][1]{$x_2$}
\includegraphics[width=0.65\linewidth]{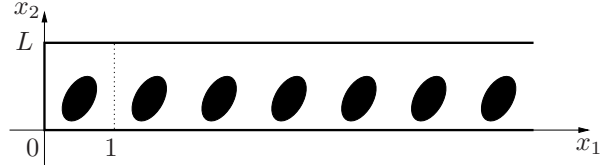}
\end{psfrags}
\caption{Dielectric permittivity of the semi-infinite waveguide.}
\label{spectrum}
\end{figure}
for some $\bar\varepsilon>0$.
On the top and bottom of $S^+$ we consider a boundary condition
\begin{equation}\label{eq:bc_top_bottom}
\gamma v(x_1,\cdot)=\zerotwo,\qquad  x_1>0
\end{equation}
with one of the following  boundary value  operators  $\gamma:H^2((0,L))\to \R^2$:
\begin{alignat*}{2}
&\mbox{Dirichlet} & \gamma_{\rm D} v &:= (v(0),v(L))^{\top}\\
&\mbox{Neumann} & \gamma_{\rm N} v &:= (v'(0),v'(L))^{\top}\\
&\mbox{mixed} & \gamma_{\rm DN} v &:= (v(0), v'(L))^{\top}\\
&\mbox{$\beta$-quasi-periodic }\quad &\gamma_{\beta} v &:= 
(e^{i\beta}v(0)-v(L),e^{i\beta}v'(0)-v'(L))^{\top},\quad
\beta\in[0,2\pi).
\end{alignat*}
To treat $\beta$-quasi-periodic boundary conditions with $\beta =0$ and $\beta =\pi$
we have to impose the symmetry condition $\epsp(x_1,x_2) = \epsp(x_1,L-x_2)$,
$x\in (0,1)\times(0,L)$ for reasons explained in {\S} \ref{sec:main_thm_proof}. \\
Furthermore, we imposed a boundary condition on the left:
\begin{equation}\label{eq:bc_left}
\ltrace v=f\qquad \mbox{with}\qquad 
\ltrace v := \theta_{\rm D} v(0,\cdot)+ \theta_{\rm N} \diffq{v}{x_1}(0,\cdot).
\end{equation}
Here $\theta_{\rm D},\theta_{\rm N}\in\C$ with $|\theta_{\rm D}|+|\theta_{\rm N}|>0$,
and we consider $\ltrace$ as an operator 
with values in a space $\lbdspace$ which depends on
$\theta_{\rm N}$ and $\gamma$ and will be defined in {\S} \ref{sec:preliminaries}.
 
To describe our condition on the behavior of the solution as $x_1\to\infty$,
which will complete the formulation of the boundary value problem, we need the
following definition: 
\begin{definition}\label{def:genFloquet_mode}
A \emph{Floquet mode} is a nontrivial solution to 
\eqref{eq:pde}, \eqref{eq:bc_top_bottom} of the form
\[
\exp(i\xi x_1)\sum_{j=0}^{m} x_1^{m-j} u_j(x_1,x_2)
\]
with $\xi\in\C$ and functions $u_j$ satisfying $u_j(x_1+1,x_2)=u_j(x_1,x_2)$
for all $(x_1,x_2)\in S^+$. 
We call $\xi$ the \emph{quasi momentum} and $m$ the \emph{order} of the
Floquet mode, assuming that $u^{(m)}\neq 0$. 
\end{definition}
Note that the quasi momentum is only determined
up to an additive integer multiple of $2\pi$. Typically we will choose 
$\Re \xi\in [-\pi,\pi)$. We complete the formulation of the boundary value problem by the
radiation condition
\begin{equation}\label{eq:radiation}
v\in H_{\gamma}^{1,+}(S^+)\quad\mbox{where}\quad
H_{\gamma}^{1,+}(S^+):= H_{\gamma}^1(S^+) \oplus \mathrm{span}\{v_1^+,\cdots,\vp_{\overline{n}}\},
\end{equation}
\end{subequations}
where $\vp_1,\dots,\vp_{\overline{n}}$ are Floquet modes with real quasi momentum, the 
choice of which is described in the following:

It is known (see \cite[Corollary 5.1.5]{B:nazarov_plamenevsky} or  
Corollary \ref{coro:propagating_modes}) 
that the wave guide supports a finite even number $2\overline{n}$ of 
linearly independent Floquet modes with real quasi momentum. 
To formulate a radiation condition we introduce the sesquilinear form
\begin{equation}\label{eq:defi_q}
q_{x_1}(v,w):= \int_0^L
\paren{\diffq{v}{x_1}(x)\overline{w(x)}-v(x)\overline{\diffq{w}{x_1}(x)}}\,dx_2,\qquad 
x_1\geq 0
\end{equation}
for solutions $v,w$ to \eqref{eq:pde} and \eqref{eq:bc_top_bottom}. 
Since $q_{x_1}$ is actually independent of $x_1$ as a consequence of Green's theorem, 
we omit the index $x_1$ in the following. If the time dependence is given by $\exp(-i\omega t)$
with $\omega>0$, then $\Im q(v,v)$ is proportional (with positive constant) to 
the energy flux through a cross section $\{x_1\}\times (0,L)$ 
(see e.g.~\cite[{\S} 5.6.3]{B:nazarov_plamenevsky} where the
sign convention $\exp(i\omega t)$ is used). From a physical mode we expect that
energy is transported to the right, so $\Im q(v,v)>0$. 
It can be shown (see \cite[Theorem 5.3.2]{B:nazarov_plamenevsky}) 
that there exists a basis 
$\{\vp_1,\cdots,\vp_{\overline{n}},v_1^-,\cdots,v_{\overline{n}}^-\}$
of the span of all Floquet modes with real quasi momentum, which consists of Floquet modes satisfying the orthogonality
and normalization conditions
\begin{subequations}\label{eq:orthonormalq}
\begin{eqnarray}
&&q(v_j^{+},v_k^{-}) = 0, \qquad \qquad \mbox{for
}j,k\in\{ 1,\dots,\overline{n}\},\\
&&q(v_j^{+},v_k^{+})=i\delta_{j,k},\qquad q(v_j^{-},v_k^{-})=-i\delta_{j,k},\qquad \mbox{for }
j,k\in\{1,\dots,\overline{n}\}.
\end{eqnarray}
\end{subequations}
(As discussed in Appendix \ref{appendix:radiation}, 
$\mathrm{span}\{\vp_1,\cdots,\vp_{\overline{n}}\}$ is not necessarily uniquely determined
by this condition for all $\omega$, but the following results hold for any choice of
the $\vp_n$.)

Each element $v\in H_{\gamma}^{1,+}(S^+)$ has a unique representation of the 
form $v= \tilde{v}+\sum_{n=1}^{\overline{n}}\alpha_n \vp_n $ with 
$\tilde{v}\in H_{\gamma}^1(S^+)$
and $\alpha_n\in \C$, and we introduce a norm on $H_{\gamma}^{1,+}(S^+)$ by
$\|v\|_{H_{\gamma}^{1,+}(S^+)}^2= \|\tilde{v}\|_{H^1_{\gamma}(S^+)}^2 + \sum_{n=1}^{\overline{n}} |\alpha_n|^2$.
%
%

Let us introduce the translation operator 
\[
(\mathcal{T}v)(x) := v(x_1+1,x_2)
\]
and assume that the Floquet modes $\vp_n$ have been chosen
such that $H^{1,+}_{\gamma}(S^+)$ is invariant under $\mathcal{T}$.
(This is always the case if $\vp_1,\dots,\vp_{\overline{n}}$ are all of order $0$.)
Let $V\subset H_{\gamma}^{1,+}(S^+)$ be the linear subspace of all weak solutions
to \eqref{eq:pde} and \eqref{eq:bc_top_bottom}. 
If the boundary value problem \eqref{dirprob} is well posed,  
i.e.\  the operator $\ltrace|_V:V\to \lbdspace$ has a bounded inverse
(see Prop. \ref{prop:wellposed} and Appendix \ref{appendix:uniqueness} for sufficient
conditions),
we can define the \emph{monodromy operator} 
\[
\mathcal{R}:= \ltrace \mathcal{T} (\ltrace|_V)^{-1}:\lbdspace \to\ \lbdspace,
\]
which maps the trace of a solution to \eqref{dirprob} to its trace at
one periodicity length to the right.
We are now in a position to formulate our  main result:
\begin{theorem}\label{theo:main}
\begin{enumerate}
	\item\label{part:translation}
There exist  Floquet modes $\vp_n$, $n\in \N$ the first $\overline{n}$ of which are
defined as described above, which form a Riesz basis of $V$.
This basis can be chosen such that $\mathcal{T}:V\to V$ is represented
by an infinite Jordan matrix, which contains at most 
a finite number of Jordan blocks of size greater than $1$ 
all of which are of finite size. 
\item
If problem \eqref{dirprob} is well posed for some choice of $\ltrace$, 
then $\{\ltrace \vp_n:n\in\N\}$ is a Riesz basis of $\lbdspace$, 
and with respect to this basis the operator 
$\mathcal{R}$ is represented by the same Jordan matrix 
as $\mathcal{T}$ in part \ref{part:translation}.
\end{enumerate}
\end{theorem}

The second part of this theorem proves the conjectures in
\cite[Remark 5.1]{JolyFlissLi} and \cite[Conjecture 3.2.52]{Fliss:09}.
(In the latter case, to prove that $R|_{\mathrm{span}
\{\ltrace \vp_n:n>\overline{n}\}}$ has spectral radius $<1$,
we also have to take into account Proposition \ref{prop:large_char_val}.)

We point out that the well-posedness assumption in the second part of Theorem \ref{theo:main}
is always satisfied for certain types of Robin trace operators $\tau$ 
(see Prop.\ \ref{prop:wellposed} and \ref{prop:uniqueness_robin}). We will use this fact
in the proof of the first part of the theorem. 
For other trace operators $\tau$ non-uniqueness may occur at certain frequencies $\omega$.

\section{Preliminaries and outline of the proof}\label{sec:preliminaries}
\subsection{Sobolev spaces on the cross section}\label{sec:spectra}
For the boundary value operators $\gamma \in\{\gamma_{\rm D},\gamma_{\rm N},\gamma_{\rm DN},\gamma_{\beta}\}$ defined in the introduction, we define the second derivative operators
\begin{align*}
&D_{\gamma}:\mathcal{D}(D_{\gamma})\to L^2((0,L)),\qquad  w\mapsto -w''\\
&\mbox{with }\mathcal{D}(D_{\gamma}) := \{w\in H^2((0,L)): \gamma w = \zerotwo \}.
\end{align*}
It is well known that these operators are positive and self-adjoint with compact 
resolvents.
Moreover, complete orthonormal systems of eigenpairs \linebreak
$\big\{\big(\tilde{\psi}_k^{(\gamma)},(\tilde{\kappa}_k^{(\gamma)})^2\big):
k\in \mathcal{I}\big\}$ are known explicitly:
$$
\begin{array}{l c c c}
\mbox{boundary condition} & \tilde{\psi}_k^{(\gamma)}(t) & \tilde{\kappa}_k^{(\gamma)} &\mathcal{I} \\  \hline
\mbox{Dirichlet} & 
   \sqrt{\frac{2}{L}}\sin\paren{\frac{\pi k}{L}t} & \frac{\pi k}{L} &\N\\ 
\mbox{Neumann} & 
   \sqrt{\frac{2}{L}}\cos\paren{\frac{\pi k}{L}t} & \frac{\pi k}{L} &\N\cup \{0\}\\
\mbox{mixed} &
   \sqrt{\frac{2}{L}}\sin\paren{\frac{\pi (2k-1)}{2L}t} & \frac{\pi (2k-1)}{2L} &\N\\
\beta\mbox{-quasi-periodic} & \sqrt{\frac{1}{L}}\exp\paren{i\frac{\beta+2\pi k}{L}t} &
\frac{\beta+2\pi k}{L} &\Z \\ \hline
\end{array}
$$
To simplify our notation we will often suppress the dependence of $\tilde{\psi}_k^{(\gamma)}$ and $\tilde{\kappa}_k^{(\gamma)}$ on $\gamma$ 
in the following.  
Sobolev spaces corresponding a boundary value operator $\gamma \in\{\gamma_{\rm D},\gamma_{\rm N},
\gamma_{\rm DN},\gamma_{\beta}\}$ can then be defined by 
$H^{s}_{\gamma}((0,L)) = \mathcal{D}((I+D_{\gamma})^{s/2})$ with norm
$\|w\|_{H^s_{\gamma}}:=\|(I+D_{\gamma})^{s/2}w\|_{L^2}$ for $s\geq 0$, 
and for $s<0$ the space $H^s_{\gamma}((0,L))$ can be defined as completion
of $L^2((0,L))$ under the norm $\|\cdot\|_{H^s_{\gamma}}$. 
More explicitly,
\begin{equation}\label{eq:normSobo}
\|w\|_{H^{s}_{\gamma}}^2 = \sum_{k\in\mathcal{I}} 
(1+\tilde{\kappa}_k^2)^s|\lsp w,\tilde{\psi}_k\rsp|^2
\end{equation}

It will be convenient to renumber the square roots of the $\tilde{\kappa}_l$'s in increasing order
such that 
\[
\sigma(D_{\gamma}) = \{\tilde{\kappa}_k^2:k\in \mathcal{I}\} = \{\kappa_l^2: l\in \N\}
\]
and $0\leq \kappa_1\leq \kappa_2\leq \dots$. This defines a bijective mapping
$\N\to \mathcal{I}$, $k\mapsto l(k)$.  
Note that
\begin{equation}\label{eq:sigma_ls}
\kappa_{2l+1} = \kappa_1 + \frac{2\pi}{L}l \qquad \mbox{and}\qquad
\kappa_{2l+2} = \kappa_2 + \frac{2\pi}{L}l
\end{equation}
for $l=1,\dots$. Moreover, we will write $\psi_k = \tilde{\psi}_{l(k)}$ for
$k\in\N$.

\subsection{Sobolev spaces on the strip}
We introduce the self-adjoint negative Laplace operators
\begin{align*}
&\tilde{D}_{\gamma} : \mathcal{D}(\tilde{D}_{\gamma})\to L^2(S),\qquad v\mapsto -\Delta v,\\
&\mbox{with }\mathcal{D}(\tilde{D}_{\gamma}) := \{v\in H^2(S):
\forall x_1\in\R\, \gamma v(x_1,\cdot) =\zerotwo\}
\end{align*}
for $\gamma \in \{\gamma_{\rm D},\gamma_{\rm N},\gamma_{\rm DN},\gamma_{\beta}\}$.
Then we define
\[
H^s_{\gamma}(S):= \mathcal{D}\paren{(I+\tilde{D}_{\gamma})^{s/2}},\qquad s\geq 0,
\]
and $\|v\|_{H^s_{\gamma}}:=\|(1+\tilde{D}_{\gamma})^{s/2}v\|_{L^2}$. 
For $s<0$, $H^s_{\gamma}(S)$ is defined as completion of $L^2(S)$ under
the norm $\|\cdot\|_{H^s_{\gamma}}$.  
It can be shown that $H^1_{\gamma_{\rm D}}(S) = H^1_0(S)$, 
$H^1_{\gamma_{\rm N}}(S) = H^1(S)$, and $H^1_{\gamma_{\beta}}(S) = 
\{v\in H^1(S):e^{i\beta}v(\cdot,0)=v(\cdot,L)\}$, and the norm $\|\cdot\|_{H^1_{\gamma}}$
is equivalent to $\|\cdot\|_{H^1}$ given by $\|v\|_{H^1}^2=\int_S(|v|^2+|\nabla v|^2)\,dx$. 
$H^s_{\gamma}(S^+)$ can be defined as the set of all restrictions
to $S^+$ of functions in $H^s_{\gamma}(S)$ with 
$\|v\|_{H^s_{\gamma}(S^+)} :=\inf\{\|\tilde{v}\|_{H^s_{\gamma}(S)}:\tilde{v}|_{S+}=v\}$. 
Moreover, the trace operators
\begin{align*}
\ltrace_{\rm D} &: H^1_{\gamma}(S^+)\to H^{1/2}_{\gamma}((0,L)),\qquad &&v\mapsto v(0,\cdot), \\
\ltrace_{\rm N} &: H^1_{\gamma}(S^+;\Delta) \to H^{-1/2}_{\gamma}((0,L)),&&
v\mapsto \diffq{v}{x_1}(0,\cdot),
\end{align*}
are well defined, continuous, and surjective, and have bounded right-inverses
(see \cite{LionsMagenesI}). (Here $\|v\|_{H^1_{\gamma}(S^+;\Delta)}^2:= 
\|v\|_{H^1_{\gamma}(S^+)}^2 + \|\Delta v\|_{L^2(S^+)}^2$.)
We choose $\lbdspace:=H^{1/2}_{\gamma}((0,L))$ if $\theta_{\rm N}=0$ and
$\lbdspace:=H^{-1/2}_{\gamma}((0,L))$ if $\theta_{\rm N}\neq 0$.  

\subsection{Floquet transform}
The Floquet transform is defined  by
\begin{align*}
\mathcal{F}:L^2(S)\to L^2\left((-\pi,\pi),L^2(\Omega)\right)\\
\mathcal{F}v(\alpha,x):=\frac{1}{\sqrt{2\pi}}\sum_{l\in\Z}v(x_1+l,x_2)e^{-i\alpha(x_1+l)}
\end{align*}
where $\Omega=\R/\Z\times(0,L)$. It is isometric 
and its inverse is given by
$v(x) = \frac{1}{\sqrt{2\pi}} \int_{-\pi}^{\pi}\mathcal{F}v(\alpha,x)e^{i\alpha x_1}\,d\alpha$,
$x\in S$ (\cite{Kuchment_Floquet}).

We will frequently use the orthonormal
bases $\{\varphi_{\mind,\nind}^{(\gamma)}:\mind\in \Z, \nind\in\N\}$ of $L^2(\Omega)$ defined by
\begin{equation}\label{genBasis}
\varphi_{\mind,\nind}^{(\gamma)}(x):=\exp(2\pi i \mind x_1) \psi_{l_2}^{(\gamma)}(x_2),\qquad x\in\Omega\,.
\end{equation}
If $v\in H^s_{\gamma}(S)$, then for all $\alpha\in [-\pi,\pi]$ the function
$\mathcal{F} v(\alpha,\cdot)$
belongs to the Sobolev space $H^s_{\gamma}(\Omega)$  defined by
$H^s_{\gamma}(\Omega) := \{u\in L^2(\Omega): \|u\|_{H^s(\Omega)}<\infty\}$ with 
\[
\|u\|_{H^s_{\gamma}(\Omega)} := 
\big(\sum_{l\in \Z\times\N} (1+|l|^2)^s |\lsp \varphi_l^{(\gamma)},u\rsp|^2\big)^{1/2}
\]
for $s\geq 0$ and as completion of $L^2(\Omega)$ under this norm for
$s<0$ (\cite{Kuchment_Floquet}). For $\alpha\in [-\pi,\pi]$ the operator 
$\Delta_{\alpha}^{(\gamma)}:H^2_{\gamma}(\Omega)\to L^2(\Omega)$
uniquely defined by the property
\[
\Delta_{\alpha} (\mathcal{F}v(\alpha,\cdot)) =
(\mathcal{F} \Delta v)(\alpha,\cdot) 
\] 
is given explicitely by
\begin{equation}\label{eq:Delta_alpha}
\begin{split}
\Delta_{\alpha} &= e^{-i\alpha x_1}\Delta  e^{i\alpha x_1}
= (\partial_{x_1}+i\alpha)^2 +\partial_{x_2}^2   
 =  \Delta  + 2i\alpha \partial_{x_1}  -\alpha^2\,.
\end{split}
\end{equation}

\subsection{Floquet modes and characteristic values of $(B_{\xi})$}\label{sec:eigenval_characteristicval}
Due to \eqref{eq:Delta_alpha} the mapping
$\alpha\mapsto \Delta_{\alpha}$ is a polynomial with coefficients in
\linebreak $\mathcal{L}(H^s(\Omega), H^{s-2}(\Omega))$ for all $s$, so in particular, it
has a holomorphic extension denoted by
\begin{equation}\label{eq:defi_Delta_xi}
\Delta_{\xi}:=\Delta  + 2i\xi \partial_{x_1}  -\xi^2,\qquad \xi\in\C. 
\end{equation}
Moreover, let us introduce the operators
\[
B_{\xi}:H^s_{\gamma}(\Omega) \to H^{s-2}_{\gamma}(\Omega)\qquad 
B_{\xi} u := \Delta_{\xi}u + \omega^2\epsp u 
\]
for $\xi\in\C$, which are well defined for any $s\in [0,2]$. 
For a review of the properties of holomorphic extensions of Floquet transformed
periodic differential operators in a much greater generality we
refer to \cite{Kuchment_Floquet}.

$\xi_0\in\C$ is called a \emph{characteristic value} of
$\xi\mapsto B_{\xi}$ if $B_{\xi_0}$ is not injective.
If $\xi_0$ is a characteristic value of $B_{\xi}$ for some
choice of $s\in [0,2]$, and $B_{\xi_0}u_0 = 0$ for 
$u_0\in H^s_{\gamma}(\Omega)\setminus\{0\}$, then 
$u_0\in H^2_{\gamma}(\Omega)$ 
by elliptic regularity results, and $\xi_0$ is a characteristic value
of $B_{\xi}$ for any parameter $s\in [0,2]$. Therefore, the set of
characteristic values does not depend on the parameter $s$, and for
studying this set we may choose $s$ at our convenience. 
The operators $B_{\xi}$ are defined such that (with $s=2$):
\begin{remark}\label{rem:char_val_eig}
The following statements are equivalent for $\xi_0\in\C$:
\begin{itemize}
\item
$\xi_0$ is a characteristic value of $(B_{\xi})$.
\item
There exists a  Floquet mode $v(x) = e^{i\xi_0x_1}u(x)$ with
$u\in H^2_{\gamma}(\Omega)$.
\end{itemize}
\end{remark}


It follows from the second statement that:
\begin{remark}\label{rem:periocity_char_val}
If $\xi\in \C$ is a characteristic value of $\xi\mapsto B_{\xi}$, then
$\xi + 2\pi l$ is a characteristic value for all $l\in\Z$. 
\end{remark}

To formulate some results by Gohberg and Sigal \cite{GohbergSigal} in Theorem \ref{theo:rouche}
below, which will serve as an essential tool in the following analysis,
we first have to recall the definition of the multiplicity of characteristic values.
Note that for all $\xi$ the operator $B_{\xi}$ is a compact perturbation of 
the operator $\Delta - 1$ in $\mathcal{L}(H^2_{\gamma}(\Omega),L^2(\Omega))$,
and that $\Delta-1$ has a  bounded inverse. Therefore, by analytic Fredholm theory \cite{steinberg:68} the set of characteristic values of $B_{\xi}$ is discrete.
For the special case at hand the definitions in \cite{GohbergSigal} simplify as follows:

\begin{definition}\label{defi:char_val}
Let $\mathcal{B}_1$ and $\mathcal{B}_2$ be Banach spaces and $\xi\mapsto B_{\xi}$
a holomorphic mapping defined on a domain in $\C$ with values in 
$\mathcal{L}(\mathcal{B}_1,\mathcal{B}_2)$.
Moreover, let $B_{\xi}= B+K_{\xi}$ where $B$ has a bounded inverse in 
$\mathcal{L}(\mathcal{B}_2,\mathcal{B}_1)$ and $K_{\xi}$ is compact for all $\xi$.

The point $\xi_0\in \C$ is called a \emph{characteristic value} 
of $(B_{\xi})$ if
there exists a holomorphic function $\xi\mapsto u_{\xi}$ (called \emph{root function})
with values in $\mathcal{B}_1$ such that $u_{\xi_0}\neq 0$ and $B_{\xi_0}u_{\xi_0}=0$. 
(Note that such root functions, which may be chosen constant, exist if and only if
$B_{\xi_0}$ is not injective.) 
The \emph{multiplicity} of a root function $(u_{\xi})$
is the order of $\xi_0$ as a root of $\xi\mapsto B_{\xi}u_{\xi}$. 
$\overline{u} \in \mathcal{B}_1$ is called an \emph{eigenvector} of $(B_{\xi})$
corresponding to $\xi_0$ if $\overline{u}=u_{\xi_0}$ for some root function
$(u_{\xi})$ of $(B_{\xi})$ corresponding to $\xi_0$.
The \emph{rank} $\rank(\overline{u})$ of an eigenvector $\overline{u}$ is defined 
as the maximum of all the multiplicities of root functions $(u_{\xi})$ with 
$\overline{u}=u_{\xi_0}$. (Under the given assumptions the geometric multiplicity
$\alpha:=\dim \ker B_{\xi_0}$ of $\xi_0$ and the ranks
of all eigenvectors $\overline{u}\in \ker B_{\xi_0}$  are finite, 
see \cite[Lemma 2.1]{GohbergSigal}.)

A \emph{canonical system of eigenvectors} of $(B_{\xi})$ corresponding to
$\xi_0$ is defined as a basis $\{u^{(1)}$,...,$u^{(\alpha)}\}$ of $\ker B_{\xi_0}$ with
the following properties: $\rank u^{(1)}$ is the maximum of
the ranks of all eigenvectors corresponding to $\xi_0$, and $\rank
u^{(j)}$ for $j=2,...,\alpha$ is the maximum of the ranks of all
eigenvectors in some direct complement of
$\mathrm{span}\{u^{(1)},...,u^{(j-1)}\}$ in $\ker B_{\xi_0}$.
The numbers $r_j=\rank u^{(j)}$ ($j=1,...,\alpha$)  are called the
\emph{partial null multiplicities} of the characteristic value $\xi_0$, 
and $\mathfrak{n}((B_{\xi});\xi_0)=(r_1,r_2,...,r_\alpha)$ the $\alpha$-tuple
of partial null multiplicities. We call
$\mathfrak{N}((B_{\xi});\xi_0)=r_1+r_2+...+r_\alpha$ the (total) null multiplicity of the
characteristic value $\xi_0$ of $(B_\xi)$.

Finally, we call  \emph{canonical Jordan chains} associated to the
characteristic value 
$\xi_0$ any sets $\{u_0^{(j)},...,u_{r_j-1}^{(j)}\}$, $1\leq j\leq\alpha$, of vectors
in $\mathcal{B}_1$ such that $\{u_0^{(1)},...,u_0^{(\alpha)}\}$ is a canonical
system of eigenvectors and $\xi_0$ is a root of order $r_j$ of
$\xi\mapsto B_\xi(\sum_{k=0}^{r_j-1} (\xi-\xi_0)^k u_k^{(j)})$ for any $1\leq j\leq\alpha$.
\end{definition}

Let $\Gamma$ be a simple, closed, rectifiable contour  
contained in the domain of analyticity of $(B_{\xi})$ and of $(B_{\xi})^{-1}$ 
and let $\xi_1,\xi_2,...,\xi_n$ be the characteristic values of $(B_{\xi})$
enclosed by $\Gamma$. (Recall that simple means that $\Gamma$ has a
continuous, bijective parametrization $\gamma:[0,\pi)\to \Gamma$.) Then we set
\begin{equation}\label{GohSigCount}
\mathfrak{N}((B_\xi);\Gamma):=\sum_{j=1}^n \mathfrak{N}((B_{\xi});\xi_j).
\end{equation}
We cite the following generalization of Rouch\'e's Theorem: 
\begin{theorem}[{\cite{GohbergSigal}}]\label{theo:rouche}
Assume that $(A_{\xi})$ satisfies the assumptions of Definition \ref{defi:char_val}
and let $\Gamma$ be a simple, closed,  rectifiable contour 
in the domain of analyticity of $(A_{\xi})$ and of $(A_{\xi})^{-1}$. 
If $\xi\mapsto S_{\xi}$ is another holomorphic function defined on the same
domain as $(A_{\xi})$ with values in 
$\mathcal{L}(\mathcal{B}_1,\mathcal{B}_2)$ and if
$$\|A^{-1}_{\xi}S_{\xi}\|<1\qquad\mbox{for all }\xi\in\Gamma,
$$
then $(A_{\xi}+S_{\xi})^{-1}$ is analytic in some neighborhood of $\Gamma$ and
\begin{equation}\label{eq:rouche_log_res}
 \mathfrak{N}((A_{\xi}+S_{\xi});\Gamma)= \mathfrak{N}((A_{\xi});\Gamma)\,. 
\end{equation}
\end{theorem}


We can now state the following generalization of Remark \ref{rem:char_val_eig}:
\begin{proposition}\label{prop:char_val_eig}
Suppose that $\xi_0$ is a characteristic value of $(B_{\xi})$ with partial
null multiplicities
$\mathfrak{n}((B_{\xi});\xi_0)=(r_1,r_2,...,r_\alpha)$. Then, 
the vector space $\mathcal{V}_{\xi_0}$ of all  Floquet modes with 
quasi momentum $\xi_0$ is a direct sum of $\mathcal{T}$-invariant subspaces 
$\mathcal{V}_{\xi_0,j}$ of dimensions $r_j\in\N$, $j=1,\dots,\alpha$. 
Moreover, there exists a basis $\{v_{j,0},\dots,v_{j,r_j-1}\}$ of $\mathcal{V}_{\xi_0,j}$ with respect to which $\mathcal{T}:\mathcal{V}_{\xi_0,j}
\to \mathcal{V}_{\xi_0,j}$ is represented
by a single $r_j\times r_j$ Jordan block, and $v_{j,k}$ are Floquet modes  of 
order $k$. Moreover, in the expansion 
\begin{equation}\label{eq:formFloquetBasis}
v_{j,k}(x) = e^{i\xi_0 x_1}\sum_{l=0}^k \frac{(ix_1)^{k-l}}{(k-l)!} u^{(j,k)}_l(x),
\qquad k=0,\dots,r_j-1 
\end{equation}
we have $u^{(j,0)}_0=\cdots=u^{(j,r_j-1)}_{r_j-1}=\tilde{u}^{(j)}_0$, 
and $\{\tilde{u}^{(1)}_0,\dots,\tilde{u}^{(\alpha)}_0\}$ is a basis of 
$\mathrm{ker}(B_{\xi_0})$.
\end{proposition}
\begin{proof}
From \cite[{\S} 3.4.3]{B:nazarov_plamenevsky} we know that there exist canonical Jordan
chains $\{\tilde{u}_0^{(j)},...,\tilde{u}_{r_j-1}^{(j)}\}$, $1\leq j\leq\alpha$ associated to $\xi_0$, and
$\mathcal{V}_{\xi_0}=\oplus_{j=1}^\alpha\mathcal{V}_{\xi_0,j}$ where
$\mathcal{V}_{\xi_0,j}=\mathrm{span}\{\tilde{v}_{j,k}:\,0\leq k\leq r_j-1\}$ and
\begin{equation}\label{eq:defi_vj}
\tilde{v}_{j,k}=e^{i\xi_0x_1}\sum_{l=0}^k\frac{(ix_1)^{k-l}}{(k-l)!}\tilde{u}^{(j)}_l.
\end{equation}
Since
\begin{align*}
(\mathcal{T}\tilde{v}_{j,k})(x)=&
e^{i\xi_0(x_1+1)}\sum_{l=0}^k\frac{(ix_1+i)^{k-l}}{(k-l)!}\tilde{u}^{(j)}_l(x)\\
=&
e^{i\xi_0}e^{i\xi_0x_1}\sum_{l=0}^k
\sum_{m=0}^{k-l}\frac{i^{m}}{m!}\frac{(ix_1)^{k-l-m}}{(k-l-m)!}
\tilde{u}^{(j)}_l(x)\\
=&
e^{i\xi_0}\sum_{m=0}^k\frac{i^{m}}{m!}
e^{i\xi_0x_1}\sum_{l=0}^{k-m}\frac{(ix_1)^{k-m-l}}{(k-m-l)!}
\tilde{u}^{(j)}_l(x)\\
=&
e^{i\xi_0}\sum_{m=0}^k\frac{i^{k-m}}{(k-m)!}\tilde{v}_{j,m}(x),
\end{align*}
the subspaces $\mathcal{V}_{\xi_0,j}$ are $\mathcal{T}$-invariant and  $\mathcal{T}$ is
represented by an upper triangular Toeplitz matrix $M_{\xi_0,j}$ with
respect to the basis $(\tilde{v}_{j,k})_{0\leq k\leq r_j-1}$ of $\mathcal{V}_{\xi_0,j}$.
Since the entries of the first upper diagonal of $M_{\xi_0,j}$ do not vanish, 
we have $(M_{\xi_0,j}-e^{i\xi_0}I)^l=0$ if and only if $l\geq r_j$, so 
$M_{\xi_0,j}$ is similar to a Jordan block $J_j$ 
of size $r_j$ with diagonal values $\exp(i\xi_0)$. It is easy to see by induction in $r_j$ 
that this similarity transform can be achieved by an upper triangular matrix 
$D$ with $1$'s on the diagonal, i.e.\ $D^{-1}M_{\xi_0,j}D=J_j$. 
Then $v_{j,k} := \tilde{v}_{j,k} + \sum_{l=0}^{k-1}D_{l,k}\tilde{v}_{j,l}$
is a Floquet mode of order $k$ of the form \eqref{eq:formFloquetBasis} 
with $u_0^{(j,k)}=\tilde{u}_0^{(j)}$, and $\mathcal{T}$ is represented by 
$J_j$ on the subspace $\mathcal{V}_{\xi_0,j}$ 
with respect to the basis $\{v_{j,0},\dots,v_{j,r_j-1}\}$ of
$\mathcal{V}_{\xi_0,j}$.
\end{proof}

\subsection{Outline of the proof of Theorem \ref{theo:main}}
We amend the finite number of right propagating Floquet modes 
$\vp_1,\dots, \vp_{\overline{n}}$ which have been described in 
section \ref{sec:the_problem} by an infinite number of decaying
Floquet modes $\vp_n$, $n=\overline{n}+1,\overline{n}+2,\dots$, 
which are chosen according to Proposition \ref{prop:char_val_eig}
for each characteristic value of $(B_{\xi})$ with positive imaginary part. 
These $\vp_n$ are arranged in increasing order of the imaginary part
of the corresponding characteristic values and will be properly
normalized. The set $\{\vp_n:n\in\mathbb{N}\}$ is a Riesz basis of $V$ 
if and only if the operator
\begin{align} \label{eq:defi_T}
T:&\, l^2(\N) \to H^{1,+}_{\gamma}(S^+) ,& (a_n)\mapsto \sum_{n=1}^{\infty}a_n \vp_n
\end{align}
is a norm isomorphism from 
$l^2(\N)$ to $V$. Our strategy is to compare $T$ with a reference operator
$\Tref$ corresponding to the case $\omega=0$, which we will refer to as 
the unperturbed case.
In the unperturbed case all interesting quantities can easily be computed
analytically. We will proceed as follows:
\begin{itemize}
	\item In section \ref{sec:char_val} we establish the existence of a countable 
	number of characteristic values and derive precise estimates on the difference of  
	characteristic values with large imaginary parts  
	in the perturbed and the unperturbed case using Theorem \ref{theo:rouche}. 
	\item In section \ref{sec:synthesis} we show that $T-\Tref$ is compact by estimating 
	the perturbation of eigenvectors of $(B_{\xi})$ for $\xi$ with large imaginary parts. 
	Moreover, we show that $\Tref$ is a norm isomorphism from $l^2(\mathbb{N})$ to
	its range. 
  \item In section \ref{sec:main_thm_proof} we show for trace operators $\tau$, which  
  satisfy the well-posedness assumption in Theorem \ref{theo:main},
  that $\tau \Tref:l^2(\mathbb{N})\to H_{\gamma}^{\tau}$ is a norm isomorphism. 
  Moreover, we show injectivity of $T$. Together with Riesz theory and 
  the well-posedness assumption this readily implies that both 
  $\tau T:l^2(\mathbb{N})\to H_{\gamma}^{\tau}$ and $T:l^2(\mathbb{N}\to V$ are norm
  isomorphisms. Moreover, by construction (see Proposition \ref{prop:char_val_eig})
  both the operators $\mathcal{R}$ and $\mathcal{T}$ are represented by Jordan 
  matrices. 
\end{itemize}

\section{Estimates of characteristic values}\label{sec:char_val}
\subsection{Characteristic values for $\omega=0$}
Recall from \S \ref{sec:spectra} that $\kappa_{n}^2$ are the eigenvalues of 
the negative Laplacian on the cross section with boundary conditions $\gamma$ and set 
\[
\xi_{m,n}:= -2\pi m + i\kappa_{n},\qquad m\in \Z,\;n\in\N\,.
\]

\begin{lemma}\label{lemm:char_val_Delta}
The characteristic values of $(\Delta_{\xi})$ are precisely the numbers $\xi_{m,n}$ 
and $\overline{\xi_{m,n}}$ (counted with their total multiplicities if these
numbers are not distinct) with $m\in Z$ and $n\in \N$. 
If $\kappa_n>0$ all partial null multiplicities are $1$ and if $\kappa_n=0$ the 
partial null multiplicity is $2$.  An eigenvector corresponding to both
 $\xi_{m,n}$ and $\overline{\xi_{m,n}}$
is given by the function $\varphi_{m,n}$ defined in \eqref{genBasis}.
\end{lemma}

\begin{proof}
Note that $-\Delta \varphi_{m,n} = ((2\pi m)^2 + \kappa_{n}^2)\varphi_{m,n} =
|\xi_{m,n}|^2\varphi_{m,n}$ for $(m,n)\in \Z\times \N$ and 
define $\widehat{u}(m,n):=\lsp u,\varphi_{m,n}\rsp$. 
Due to \eqref{eq:defi_Delta_xi} we have
\begin{equation}\label{eq:Delta_hat}
\begin{split}
-\widehat{(\Delta_\xi u)}(l) 
&= 
- (\Delta + 2i\xi (2\pi i m) -\xi^2)\widehat{u}(l) 
= (|\xi_l|^2  + 2\xi 2\pi m + \xi^2)\widehat{u}(l)\\
&= (\xi-\xi_{l})(\xi-\overline{\xi_{l}})\,\widehat{u}(l),\qquad 
l=(m,n)\in \Z\times \N.
\end{split}
\end{equation}
Since $\{\varphi_l:l\in \Z\times \N\}$ is an orthonormal basis
of $L^2(\Omega)$, this shows that the problem separates into a countable set
of scalar problems indexed by $l=(m,n)$, each of which has the two simple
characteristic values $\xi_l$ and $\overline{\xi_l}$ if $\kappa_n>0$ and
the characteristic value $0$ of multiplicity $2$ if $\kappa_n=0$.
\end{proof}

Since the characteristic values are always
$2\pi$-periodic, also for $\omega>0$ (see Remark \ref{rem:periocity_char_val}), 
it suffices to
study the characteristic values in the strip $\{z\in\C:-\pi \leq \Re z <\pi\}$.
Therefore we define 
\[
\xiup_n=i\kappa_n= \xi_{0,n},\qquad n\in\N
\]
as ``reference characteristic values''. Note that $\xiup_n$ corresponds to 
the physical exponentially decaying Floquet mode 
$\exp(i\xiup_n x_1)=\exp(-\kappa_n x_1)$ 
if $\kappa_n>0$ whereas $\overline{\xiup_n}$ corresponds to the unphysical exponentially growing Floquet mode $\exp(i\overline{\xiup_n}x_1)= \exp(\kappa_n x_1)$.

\subsection{Estimates of ``large'' characteristic values for Dirichlet boundary conditions}
For the sake of clarity we first prove estimates on the location of 
characteristic values for the case of Dirichlet boundary conditions 
$\gamma = \gamma_{\rm D}$ where
$\kappa_n = \frac{\pi n}{L}$ before treating general boundary conditions. 
The main tool will be Theorem \ref{theo:rouche}. 

We define open discs $\mathcal{D}_n$ and closed rectangles $\mathcal{R}_n$ by 
\begin{align}
\label{eq:defi_disk}
\mathcal{D}_n &:= \left\{z\in \C: |z-\xiup_n|< \frac{2\omega^2\bar\varepsilon}{\kappa_{\nind}+\kappa_{\nind-1}}\right\},\\
\mathcal{R}_n &:= \left\{z\in \C: |\Re z|\leq \pi, |\Im z-\kappa_n| \leq \frac{\pi}{2L}\right\}\,.
\end{align}
The radii of these discs are chosen such that $\|\Delta_{\xi}^{-1}\omega^2\eps_p\|_{\calL(L^2(\Omega))}<1$ for all 
$\xi \in \partial\mathcal{D}_n$ as shown later. Moreover, we choose $N\in \N$  
such that all disks $\mathcal{D}_n$ with $n\geq N$ are mutually disjoint and contained in the strip $\{z\in\C:|\Re z|<\pi\}$, i.e.\ 
\begin{equation}\label{eq:setS_N}
N\geq \min\braces{\nind\in\N :\frac{\kappa_{\nind}+\kappa_{\nind-1}}{2}> 
\frac{\omega^2\bar\varepsilon\max(1,L)}{\pi}}\,.
\end{equation}

\begin{proposition}\label{prop:Diri_char_val_eig}
For Dirichlet boundary conditions we have 
\[
\mathfrak{N}((B_\xi),\partial D_n) = 1=\mathfrak{N}((B_\xi),\partial R_n) 
\qquad \mbox{for all }n\geq N\,.
\]
In particular, all characteristic values of $(B_{\xi})$ in the half-strip
$\{z\in\C: -\pi\leq \Re z< \pi, \Im z>\kappa_N-\pi/(2L)\}$ are simple
and contained in one of the disks $\mathcal{D}_n$. 
\end{proposition}

\begin{proof}
For each $n\geq N$ we apply Theorem \ref{theo:rouche} with 
$A_{\xi} = \Delta_{\xi}\in \calL(L^2(\Omega),H^{-2}_{\gamma}(\Omega))$
and $S_{\xi} = B_{\xi}-\Delta_{\xi} =  \omega^2\eps_p$. We will show that 
\begin{equation}\label{eq:aux_loc_char_val}
\|\Delta_{\xi}^{-1} \omega^2\eps_p\|_{\calL(L^2(\Omega))}<1
\qquad \xi\in \mathcal{R}_n \setminus \mathcal{D}_n \,.
\end{equation}
Then  Theorem \ref{theo:rouche} is applicable with 
$\Gamma=\partial \mathcal{D}_n$ and yields 
$\mathfrak{N}((B_\xi),\partial D_n) = \mathfrak{N}((\Delta_\xi),\partial D_n) = 1$
(with the help of Lemma \ref{lemm:char_val_Delta}), and it follows from 
the Neumann series or another application
of Theorem \ref{theo:rouche} with $\Gamma = \partial \mathcal{R}_n$ that there are
no characteristic values of $(B_{\xi})$ in $\mathcal{R}_n \setminus \mathcal{D}_n$. This proves all claims. 

Since 
$\|\omega^2\varepsilon_p\|_{\calL(L^2(\Omega))}
= \omega^2\bar\varepsilon$, eq.~\eqref{eq:aux_loc_char_val} holds true 
if we can show that
\[
\|\Delta_{\xi}^{-1}\|_{\calL(L^2(\Omega))}
<\frac{1}{\omega^2\bar\varepsilon},
\qquad \xi\in \mathcal{R}_n \setminus \mathcal{D}_n\,.
\]
By virtue of \eqref{eq:Delta_hat} this is equivalent to
\[
\sup\braces{\frac{1}{|\xi-\xi_{l}|\,|\xi-\overline{\xi_{l}}|}  
: l\in \Z\times \N}
<\frac{1}{\omega^2\bar\varepsilon},\qquad 
\xi\in \mathcal{R}_n \setminus \mathcal{D}_n \,.
\]
Note that 
\begin{align}\label{eq:alt_def_R}
\mathcal{R}_n =\braces{\xi \in\C: \forall l\in \Z\times \N \;|\xi-\xi_{0,n}|\leq |\xi-\xi_l|}\,,
\end{align} 
so $|\xi-\xi_l|\geq |\xi-\xi_{0,n}|\geq 
\frac{2\omega^2\bar\varepsilon}{\kappa_{\nind}+\kappa_{\nind-1}}$ for all
$\xi\in \mathcal{R}_n \setminus \mathcal{D}_n$. Together
with the inequality $|\xi-\overline{\xi_l}|\geq |\Im \xi-\Im \overline{\xi_l}|
>\Im\xi \geq \frac{1}{2}(\kappa_n+\kappa_{n-1})$ we obtain the desired estimate
\begin{align*}
|\xi-\xi_l|\,|\xi-\overline{\xi_l}|>& 
\omega^2 \bar\varepsilon\qquad \mbox{for all } l\in \Z\times \N
\end{align*}
(see Fig.~\ref{subfig:diri}). 
Since the supremum over $l$ above is obviously attained, we get a strict inequality.
\end{proof}

\subsection{Estimates of ``large'' characteristic values for general boundary conditions}
For general lateral boundary conditions, in particular quasi-periodic boundary
conditions, the eigenvalues $\kappa_n$ may come in pairs which are arbitrarily
close together or coincide. Recall, however, that we have $\kappa_{n+2} = 
\kappa_n+\frac{2\pi}{L}$ for all $n\in\N$ and for all trace operators 
$\gamma$. We assign to each index $n$ its group $G(n)$ as follows:
$G(n):=\{n\}$ if $\kappa_{n+1}-\kappa_n = \frac{\pi}{L}$, i.e.\ if the $\kappa_n$ are
equidistant,  which is the case for $\gamma\in\{\gamma_{\rm D},\gamma_{\rm N},\gamma_{\rm DN}\}$. If $\kappa_{n+1}-\kappa_n<\frac{\pi}{L}$, then 
$G(n):=\{n,n+1\}$ and if $\kappa_{n+1}-\kappa_n>\frac{\pi}{L}$, then
$G(n):=\{n,\max(n-1,1)\}$.

To each group we assign a set 
\begin{align}\label{eq:defi_RGn}
\mathcal{R}_{G(n)}:=\{\xi\in\C:\forall l\in \Z\times \N\; \min_{n'\in G(n)}|\xi-\xiup_n| \leq |\xi-\xi_l|\}\,,
\end{align}
which is a closed rectangle since the $\xi_l$ form a rectangular grid
(see Fig.~\ref{subfig:set_S}). 
As remarked in \eqref{eq:alt_def_R}, these rectangles coincide with the ones defined
in the previous subsection if $\gamma=\gamma_{\rm D}$. Moreover, we introduce
open disks, which again coincide with those of the previous subsection for
$\gamma=\gamma_{\rm D}$:
\[
\mathcal{D}_n := \left\{z\in \C: |z-\xiup_n|< \frac{\omega^2\bar\varepsilon}{\inf_{z\in \mathcal{R}_{G(n)}}\Im z}\right\}
\]
\begin{psfrags}
\psfrag{0}[][][0.8]{$0$}
\psfrag{Rez}[][][0.8]{$\Re z$}
\psfrag{Imz}[][][0.8]{$\Im z$}
\psfrag{p}[][][0.8]{$\pi$}
\psfrag{2p}[][][0.8]{$2\pi$}
\psfrag{-p}[][][0.8]{$-\pi$}
\psfrag{-2p}[][][0.8]{$-2\pi$}
\begin{figure}[ht]
\subfigure[\label{subfig:diri} Proof of Proposition \ref{prop:Diri_char_val_eig}]{
\psfrag{D}[][][0.8]{$\mathcal{D}_n$}
\psfrag{Rn}[][][0.8]{$\mathcal{R}_n$}
\psfrag{xi}[][][0.8]{$\xi$}
\psfrag{Rn}[][][0.8]{$\mathcal{R}_n$}
\psfrag{il2}[][][0.8]{$\xi_{0,n}$}
\psfrag{mil2}[][][0.8]{$\overline{\xi_{0,n}}$}
\includegraphics[width=0.39\textwidth]{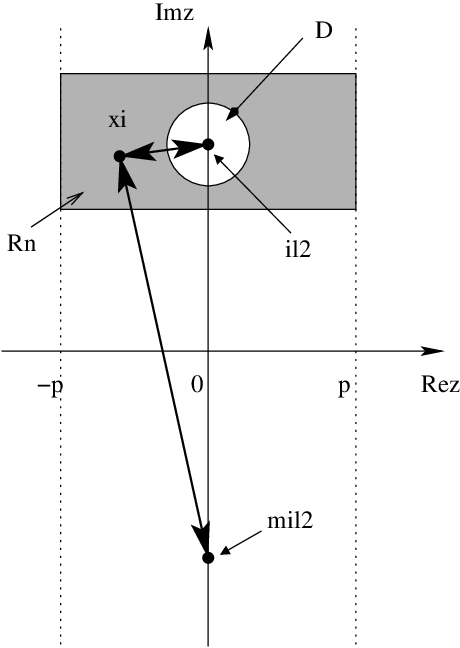}}
\subfigure[\label{subfig:set_S}Proposition \ref{prop:char_val_eig}]{
\includegraphics[width=0.6\textwidth]{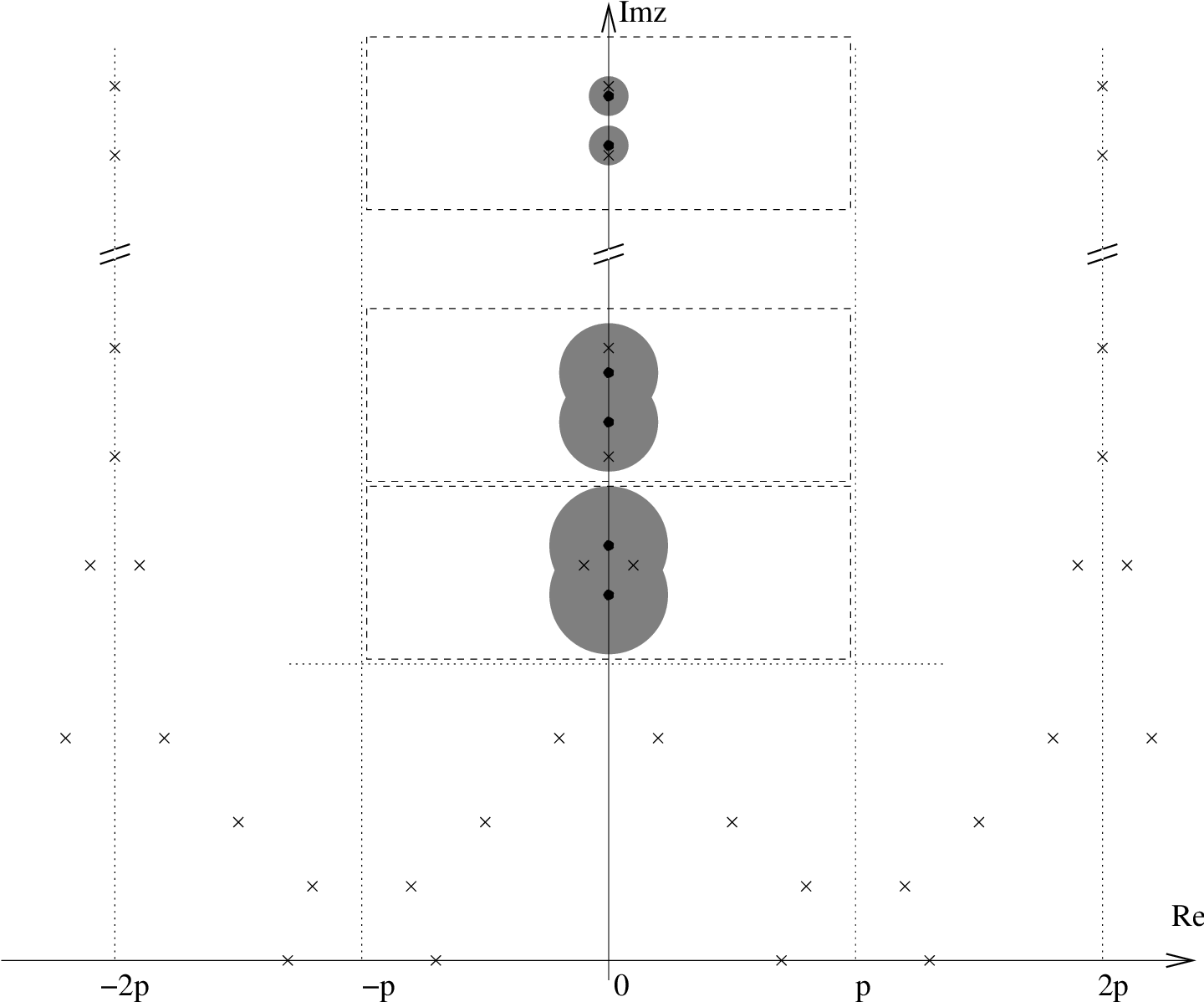}}\\
\caption{
Panel (a): An important ingredient of the proof dictating our definition
of the geometry are uniform lower bounds on products of the distances indicated
by fat double arrows.\newline
Panel (b): The shaded area shows the set $\mathcal{S}$ defined in 
\eqref{eq:setS} in the case of quasi-periodic boundary conditions.
The crosses indicate the characteristic values of $(B_\xi)$,
and the dots indicate some of the reference 
characteristic values  $(\Delta_\xi)$. 
The dashed lines show the boundaries of rectangles $\mathcal{R}_{G(n)}$ defined
in \eqref{eq:defi_RGn}.}
\end{figure}
\end{psfrags}

\begin{proposition}\label{prop:large_char_val}
Define the set
\begin{equation}\label{eq:setS}
\mathcal{S}:= \bigcup_{\nind\geq N} \mathcal{D}_n
\qquad \mbox{with}\qquad 
N\geq \min\braces{\nind\in\N :\inf_{z\in \mathcal{R}_{G(n)}}\Im z > 
\frac{\omega^2\bar\varepsilon\max(1,L)}{\pi}}
\end{equation}
(see Fig.~\ref{subfig:set_S}).  Then 
all characteristic values of $(B_{\xi})$ in  $\{z\in\C:\Re z\in [-\pi,\pi], 
\Im z\geq \inf_{z\in R_{G(N)}}\Im z\}$ are contained in $\mathcal{S}$ and 
for each connected component $\Scomp$  of $\mathcal{S}$ we have
\begin{equation}
\mathfrak{N}((B_{\xi});\partial \Scomp) 
= \mathfrak{N}((\Delta_{\xi});\partial \Scomp) 
= \#\{n\in \N:\xiup_{n}\in\Scomp\}\,.
\end{equation}
\end{proposition}


\begin{proof}
The proof is analogous to that of Proposition \ref{prop:Diri_char_val_eig}. 
Note that $N$ is chosen
such that $\mathcal{S}\subset \{z\in\C:|\Im z|<\pi\}$ and 
$\bigcup_{n\in G(n)} \mathcal{D}_n \subset \mathcal{R}_{G(n)}$ for
all $n\geq N$. By the same arguments we can show that 
$\|\Delta_{\xi}^{-1} \omega^2\eps_p\|_{\calL(L^2(\Omega))}<1$ for all $\xi \in \mathcal{R}_{G(n)}\setminus \Scomp$ if 
$\Scomp\subset \bigcup_{n\in G(n)} \mathcal{D}_n$. 
\end{proof}

\subsection{Number of ``small'' characteristic values}
Only for very small $\omega$ the argument of Proposition \ref{prop:large_char_val} 
can work for all $n$ since for large $\omega$ some of the discs $\mathcal{D}_n$
overlap with their neighbors. In the following we will determine the number 
of characteristic values with ``small'' imaginary parts.
First we cite the following symmetry result for the set of characteristic values:
\begin{proposition}[{\cite[Theorem 5.3]{GohbergSigal}}]\label{prop:symm} 
For real-valued $\epsilon$ the set of characteristic values of $(B_{\xi})$
is symmetric with respect to the real axis, and moreover
$\mathfrak{N}((B_{\xi});\bar\xi_0)=\mathfrak{N}((B_{\xi});\xi_0)$ for all 
$\xi_0\in\C$. 
\end{proposition}

To count the number of these ``small'' characteristic values in the strip
$\{z\in\C:-\pi\leq \Re(z)<\pi\}$ we surround them by a contour containing 
none of them.
By Proposition \ref{prop:large_char_val} the segment $[z_N,z_N+2\pi]$ with
$z_N:=-\pi+i\inf_{z\in\mathcal{R}_{G(N)}}\Im z$ contains no characteristic 
value.

\begin{figure}[!ht]
\subfigure[\label{subfig:contour_small_char_val}Proposition \ref{prop:firstcharvect}]{
\begin{psfrags}
\psfrag{xn}[][][1]{$\xi_{N-1}^+$}
\psfrag{xn1}[][][1]{$\xi_{N}^+$}
\psfrag{kn}[][][1]{$\xiup_{N-1}$}
\psfrag{kn1}[][][1]{$\xiup_{N}$}
\psfrag{overkn}[][][1]{$\overline{\xiup_{N-1}}$}
\psfrag{overknp}[][][1]{$\overline{\xiup_{N}}$}
\psfrag{minusxin}[][][1]{$\xi_{N-1}^-$}
\psfrag{minusxinp}[][][1]{$\xi_{N}^-$}
\psfrag{p}[][][1]{$\pi$}
\psfrag{-p}[][][1]{$-\pi$}
\psfrag{2p}[][][1]{$2\pi$}
\psfrag{-2p}[][][1]{$-2\pi$}
\psfrag{0}[][][1]{$0$}
\psfrag{zn}[][][1]{$z_N$}
\psfrag{zn2p}[][][1]{$z_N+2\pi$}
\psfrag{znb}[][][1]{$\overline{z_N}$}
\psfrag{P}[][][1]{$P$}
\psfrag{Imz}[][][1]{$\Im z$}
\psfrag{Rez}[][][1]{$\Re z$}
\includegraphics[width=0.55\linewidth]{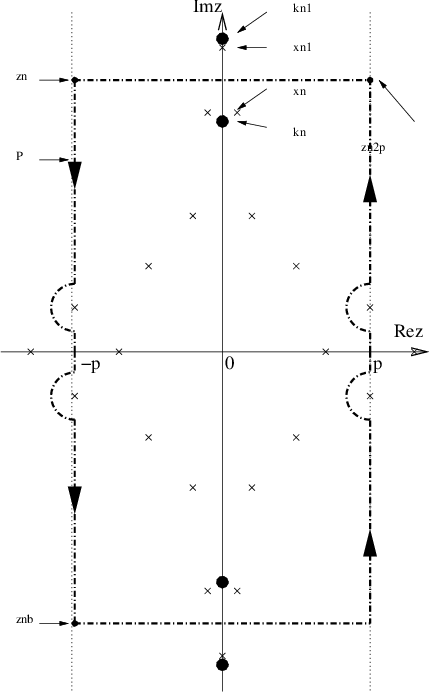}
\end{psfrags}}
\subfigure[\label{subfig:FloquetModes}Lemma \ref{lemm:FloquetModes}]{\begin{psfrags}
\psfrag{S}[][][1]{$\Scomp$}
\psfrag{xi}[][][1]{$\xi$}
\psfrag{xin+1}[][][1]{$\xi_{0,n+1}$}
\psfrag{xin-1}[][][1]{$\xi_{0,n-1}$}
\psfrag{xi-1n}[][][1]{$\xi_{-1,n}$}
\psfrag{delta}[][][1]{$\delta_{\gamma}$}
\psfrag{p}[][][1]{$\pi$}
\psfrag{-p}[][][1]{$-\pi$}
\psfrag{2p}[][][1]{$2\pi$}
\psfrag{-2p}[][][1]{$-2\pi$}
\psfrag{0}[][][1]{$0$}
\psfrag{Imxi}[][][1]{$\Im z$}
\psfrag{Rexi}[][][1]{$\Re z$}
\includegraphics[width=0.45\linewidth]{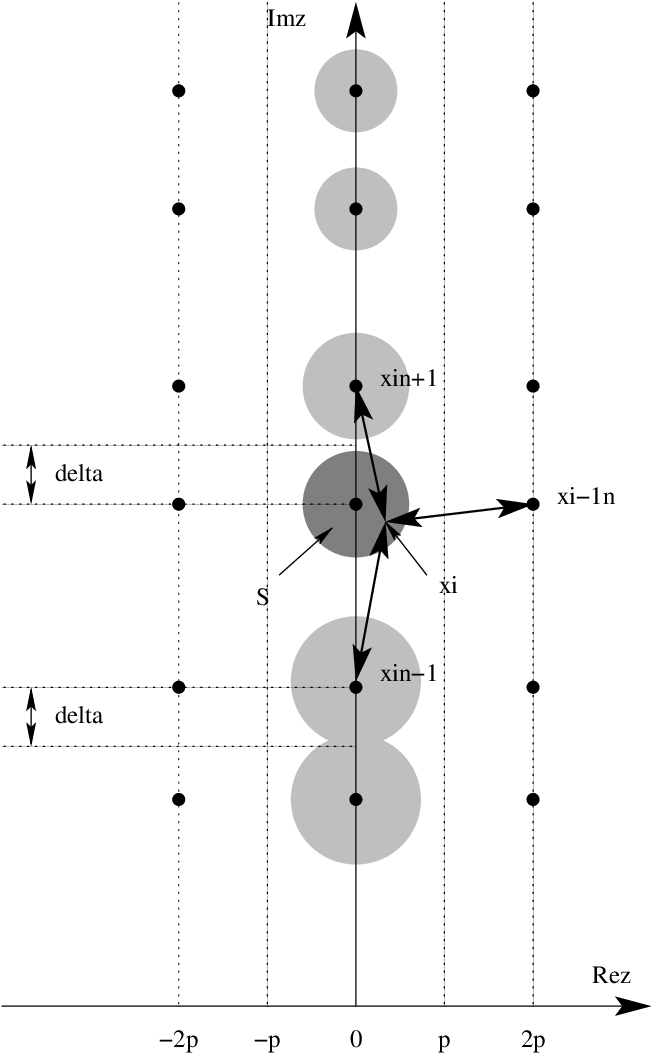}
\end{psfrags}}
\caption{
Panel (a): The dashed line indicated the
integration contour in the proof of Proposition \ref{prop:firstcharvect}. 
Crosses indicate characteristic values of $(B_{\xi})$, circles mark
four of the ``reference'' characteristic values of $(\Delta_{\xi})$.\newline
Panel (b): The proof of Lemma \ref{lemm:FloquetModes} involves lower bounds
of the distances of points $\xi$ in a connected component $\Scomp$ of the set 
$\mathcal{S}$ (the dark shaded region) 
to all ``reference'' characteristic values $\xi_{m,n}$ not
contained in $\Scomp$, see \eqref{eq:aux_as_estim1}.}
\end{figure}

\begin{proposition}\label{prop:firstcharvect}
Choose $z_N$ as above and any complex path $P$ from $z_N$  to $\bar{z}_N$ 
such that $\Gamma:=[z_N+2\pi,z_N]\cup P \cup[\bar{z}_N,\bar{z}_N+2\pi]\cup 
-(2\pi+P)$
encloses precisely all characteristic values of $(B_\xi)$ in the rectangle
$\mathcal{R}:=\{z\in\C: \Re z\in [-\pi,\pi), |\Im z|<|\Im z_N|\}$ 
(see Fig.~\ref{subfig:contour_small_char_val}).  Then
$$\mathfrak{N}((B_\xi);\Gamma)=2N.$$
\end{proposition}

\begin{proof}
Let us define $B(\xi,\mu):=\Delta_\xi + \mu\epsp$ for $\mu\in [0,\omega^2]$. 
From Proposition \ref{prop:large_char_val} we can deduce that the segments
$[z_N;z_N+2\pi]$ and $[\bar{z}_N+2\pi,\bar{z}_N]$ contain
no characteristic values of $\xi\mapsto B(\xi,\mu)$ for any 
$\mu\in[0,\omega^2]$. For a given $\mu \in [0,\omega^2]$ it follows 
from the discreteness of the set of characteristic values of 
$\xi\mapsto B(\xi,\mu)$ that the segment $[z_N,\overline{z}_N]$ contains
at most a finite number of characteristic values, and there exists $\delta>0$
such that $\delta$-balls around these characteristic values contain no
further characteristic values. We deform the straight line 
$[z_N,\overline{z}_N]$ in $\delta$-neighborhoods of the characteristic
values to left semi-circles as shown in Fig.~\ref{subfig:contour_small_char_val} to obtain a contour $P_{\mu}$. Because of the $2\pi$ periodicity of the characteristic values, 
the imaginary parts of the
characteristic values on $[z_N,\overline{z}_N]$ and on $2\pi+[\overline{z}_N,z_N]$ coincide, and replace the segment $2\pi+[\overline{z}_N,z_N]$ by $-(2\pi+P_{\mu})$. This yields a contour 
$\Gamma_{\mu}:=[z_N+2\pi,z_N]\cup P_{\mu} \cup[\bar{z}_N,\bar{z}_N+2\pi]\cup 
-(2\pi+P_{\mu})$ 
which contains precisely all characteristic values of
$B(\cdot,\mu)$ in the rectangle $\mathcal{R}$. Moreover, 
$\Gamma_{\omega^2}$ coincides with the contour $\Gamma$ in the proposition. With this construction the function 
$\overline{N}:[0,\omega^2]\to \{0,1,2,\dots\}$,
\[
\overline{N}(\mu):=\mathfrak{N}((B_\xi);\Gamma_{\mu})
\]
is well-defined, and in particular it does not depend on the
choice of $\Gamma_{\mu}$. Moreover, it follows from Theorem \ref{theo:rouche} 
with $A_{\xi} = B(\xi,\mu)$ and $S_{\xi} =\tilde{\delta} \epsp$
such that $|\tilde{\delta}|<1/\left(\overline{\varepsilon}\,\max_{\xi\in \Gamma_{\mu}} \|A_{\xi}^{-1}\|_{\mathcal{L}(L^2(\Omega))}\right)$  
that $\overline{N}$ is constant in a neighborhood of each 
$\mu\in [0,\omega^2]$. 
Hence, $\overline{N}$ is constant on whole interval $[0,\omega^2]$ and 
\[
\mathfrak{N}((B_\xi);\Gamma)= \overline{N}(\omega^2) = \overline{N}(0)
= \mathfrak{N}((\Delta_\xi);\Gamma_0) = 2N.
\]
\end{proof}

We mention that Proposition \ref{prop:firstcharvect} yields an independent proof of the following
well-known result (see \cite[Corollary 5.1.5]{B:nazarov_plamenevsky}):
\begin{corollary}\label{coro:propagating_modes}
The dimension of the space $\mathcal{V}_{\rm p}$ spanned by Floquet modes with real 
quasi momentum is finite and even. 
\end{corollary}

\begin{proof}
Due to Proposition \ref{prop:symm} 
the contour $\Gamma$ in Proposition \ref{prop:firstcharvect} encloses 
the same number $\overline{m}$ of characteristic values with positive and with
negative imaginary parts. Together with Proposition \ref{prop:char_val_eig} we find that 
$\dim \mathcal{V}_{\rm p}= 2N-2\overline{m}$ is finite and even.
\end{proof}

\subsection{Summary}\label{sec:summary}
Let us summarize our results on the characteristic values of 
$(B_{\xi})$ in the strip $\{z\in\C:-\pi \leq \Re z<\pi\}$ and 
introduce some notation for the following sections:
\begin{itemize}
	\item The set of characteristic values of $(B_{\xi})$ in the strip 
	$\{z\in\C:-\pi \leq \Re z<\pi\}$ is countable and symmetric with respect
	to the real axis.
	\item An even number $2\overline{n}$ of these characteristic values
	(counted with their total null multiplicity) lies on the real axis. 
	The corresponding Floquet modes are propagating. $\overline{n}$ linearly 
	independent ``outgoing'' Floquet modes $\vp_1,\dots \vp_{\overline{n}}$
	are selected as described in section \ref{sec:the_problem} or 
	\ref{appendix:radiation}, and the corresponding
	characteristic values (or quasi-momenta) are denoted by
	$\xi_1^+,\dots,\xi_{\overline{n}}^+$. 
	\item We arrange the characteristic values with positive imaginary parts 
	(counted with their total null multiplicities) in non-decreasing order of their 
	imaginary parts and labeled them 
	$\xi_{\overline{n}+1}^+,\xi_{\overline{n}+2}^+,\dots$.
	Corresponding exponentially decaying Floquet modes are chosen as in Proposition 
	\ref{prop:char_val_eig}) and denoted by 
	$\vp_{\overline{n}+1}, \vp_{\overline{n}+2},\dots$. 
	For $n>N$ the characteristic values $\xi_n^+$ are close to the reference  
	values $\xiup_n$ and belong to the set $\mathcal{S}$ sketched in 
	Fig.~\ref{subfig:set_S}. 
	\item 
	For Dirichlet and Neumann lateral boundary conditions the characteristic
	values $\xi_n^+$ for $n>N$ are simple. For other boundary conditions the
	total null multiplicity of $\xi_n^+$ with $n>N$ is always $\leq 2$, and except
	for quasi-periodic boundary conditions with $\beta\in\{0,\pi\}$ at most
	a finite number of them is not simple.  
\end{itemize}

\section{Properties of the operator $T$}\label{sec:synthesis}
\subsection{Reference operator $\Tref$}
To study the operator $T:(a_n)\mapsto \sum_{n=1}^{\infty}a_n \vp_n$ 
in \eqref{eq:defi_T}, we introduce a reference operator $\Tref$
using the Floquet modes
\begin{equation}\label{eq:defi_vr}
\vr_n(x):= \paren{\kappa_n+\frac{1}{2\kappa_n}}^{-1/2} e^{-\kappa_nx_1}\psi_n(x_2),
\qquad n\in \N\quad \mbox{if } \kappa_n>0
\end{equation}
corresponding to $\omega =0$. The normalization constant has been chosen such
that $\|\vr_n\|_{H^1_{\gamma}(S^+)} = 1$. 
In the special case $\kappa_n=0$, which can 
only occur for $n=1$, we set $\vr_1(x):=e^{-x_1}\psi_1(x_2)$. In this case 
$\vr_1$ is not a solution to \eqref{eq:pde} with $\omega=0$,
but we will not need this property. 
$\Tref$ is defined by 
\begin{align*}
\Tref:&\, l^2(\N) \to H^{1,+}_{\gamma}(S^+),& (a_n)\mapsto \sum_{n=1}^{\infty}a_n \vr_n,
\end{align*}
and has the following properties:

\begin{lemma}
\label{lemm:T0}
\begin{enumerate}
\item 
$\Tref$ is well defined and
isometric with $\Tref(l^2(\N))\subset H^{1}_{\gamma}(S^+)$, i.e.\  $\|\Tref(a_n)\|_{H^{1}_{\gamma}} = \|(a_n)\|_{l^2}$
for all $(a_n)\in l^2(\N)$.
\item
$\ltrace \Tref:l^2(\N)\to \lbdspace$ is a Fredholm operator with index $0$. 
\end{enumerate}
\end{lemma}

\begin{proof}
\emph{Part 1:}  
The assertion follows from the fact that $\{\vr_n:n\in\N\}$ is an orthonormal
system in $H^{1,+}_{\gamma}(S^+)$.

\emph{Part 2:} Define $\tilde{\psi}_n:= \paren{\kappa_n+\frac{1}{2\kappa_n}}^{-1/2}\psi_n$ 
if $\kappa_n>0$ and $\tilde{\psi}_n:=\psi_n$ else. 
Then 
$(\ltrace \Tref)(a_n) = \sum_{n=1}^{\infty}(\theta_{\rm D}-\kappa_n\theta_{\rm N}) a_n\tilde{\psi}_n$. If $\theta_{\rm N}=0$, the assertion
follows from the fact that $\{\tilde{\psi}_n:n\in\N\}$ is a 
complete orthogonal system in $\lbdspace = H^{1/2}_{\gamma}((0,L))$ and 
$\sup_{n\in\N}\|\tilde{\psi}_n\|_{H^{1/2}}/
\inf_{n\in\N}\|\tilde{\psi}_n\|_{H^{1/2}}<\infty$. 
To treat the case $\theta_{\rm N}\neq 0$ note that $\{\tilde{\psi}_n:n\in\N\}$ is 
also a complete orthogonal system in $\lbdspace = H^{-1/2}_{\gamma}((0,L))$
and that 
$\sup_{n\in\N}((1+\kappa_n^2)^{1/2}\|\tilde{\psi}_n\|_{H^{-1/2}})/
\inf_{n\in\N}((1+\kappa_n^2)^{1/2}\|\tilde{\psi}_n\|_{H^{-1/2}})<\infty$.
If none of the coefficients $\theta_{\rm D}-\kappa_n\theta_{\rm N}$,
$n\in\N$ vanishes, this immediately implies that 
$\ltrace \Tref:l^2(\N)\to  H^{-1/2}_{\gamma}((0,L))$ is bounded and
boundedly invertible. If $\theta_{\rm D}-\kappa_n\theta_{\rm N}=0$
precisely for $n=n_0$, then the operator 
$(a_n)\mapsto (\ltrace \Tref)(a_n) +a_{n_0}\psi_{n_0}$, which
is a rank-1 perturbation of $\ltrace \Tref$, is boundedly invertible
and hence $\ltrace \Tref$ is Fredholm with index $0$. The case that
two of the coefficients $\theta_{\rm D}-\kappa_n\theta_{\rm N}$ vanish
can be treated analogously, and because of the form \eqref{eq:sigma_ls}
of the $\kappa_n$ no more than two coefficients can vanish.
\end{proof}

\subsection{Estimates on perturbation of eigenvectors}
In the following estimates on the perturbation of eigenvectors we will
make all constants explicit. To do so, we first have to introduce
for each connected component $\Scomp$ of $\calS$ in \eqref{eq:setS} 
an $L^2$-orthogonal projection operator
\begin{equation}\label{eq:defiPScomp}
P_{\Scomp} u:= 
\sum_{n\in I_{\Scomp}} \lsp u,\varphi_{0,n}\rsp \varphi_{0,n}
\qquad \mbox{with}\qquad
I_{\Scomp}:=\{n\in \N:\xiup_n\in\Scomp\}\,
\end{equation}
and define the quantity 
\[
\delta_{\gamma}:=\frac{1}{2}\inf\{|\kappa_n-\kappa_{n'}|:n,n'\in \N, \kappa_n\neq
\kappa_{n'}\}\,.
\]
It is easy to see from the explicit values of the $\kappa_n$ given in 
\S~\ref{sec:spectra} that $\delta_{\gamma}=\pi/(2L)$ for 
$\gamma\in\{\gamma_{\rm D}, \gamma_{\rm N},\gamma_{\rm DN}\}$, 
$\delta_{\gamma_\beta}=\min_{l\in\N}|l\pi-\beta|/L$ if 
$\beta\notin\{0,\pi\}$, and $\delta_{\gamma_\beta}=\pi/L$ for
$\beta\in\{0,\pi\}$. Note that for a connected component $\Scomp$ of 
the set $\mathcal{S}$ defined in Proposition \ref{prop:char_val_eig} we have
\begin{equation}\label{eq:use_of_delta_gamma}
|\Im\xi- \Im \xi_{0,n}| > \delta_{\gamma}\qquad \mbox{for all }
\xi\in \Scomp, \xi_{0,n}\notin \Scomp 
\end{equation}
(see Fig.~\ref{subfig:FloquetModes}).

\begin{lemma}[estimates of eigenvectors]\label{lemm:FloquetModes}
For a connected component $\Scomp$ of $\calS$ in \eqref{eq:setS} 
define $\kappa_{\Scomp}:=\inf_{\xi\in\Scomp}\Im \xi$ and 
recall the notation \eqref{eq:defiPScomp}. 
Then all eigenvectors $u^*$ corresponding
to characteristic values in $\Scomp$ satisfy the estimates
\begin{subequations}
\begin{align}
\label{eq:as_estim_u}
\|u^*-P_{\Scomp}u^*\|_{L^2(\Omega)} 
&\leq  \frac{\omega^2 \bar\eps}
{\min(\delta_{\gamma},\pi) \kappa_{\Scomp}}\|u^*\|_{L^2(\Omega)},\\
\label{eq:as_estim_u1}
\|\partial_{x_1}u^*\|_{L^2(\Omega)}
&\leq \frac{2\omega^2 \bar\eps}{\kappa_{\Scomp}} \|u^*\|_{L^2(\Omega)},\\
\label{eq:diri_trace}
\|(u^*-P_{\Scomp}u^*)(0,\cdot)\|_{L^2((0,L))}
&\leq C \frac{\omega^2\bar\eps}{\kappa_{\Scomp}}\|u^*\|_{L^2(\Omega)},\\
\label{eq:neum_trace}
\Big\|\diffq{u^*}{x_1}(0,\cdot)\Big\|_{L^2((0,L))}
&\leq \frac{2\omega^2\bar\eps}{\sqrt{3}}\|u^*\|_{L^2(\Omega)}.
\end{align}
\end{subequations}
with $C:=(\sum_{\mind \in\Z} \frac{1}{\pi^2\mind ^2 + \delta_{\gamma}^2})^{1/2}$.
\end{lemma}

\begin{proof} Let $\xi\in\Scomp$ be a characteristic value of $(B_{\xi})$
with eigenvector $u^*$. It follows from \eqref{eq:Delta_hat} that
$(\xi-\xi_l)(\xi-\overline{\xi_l}) \widehat{u^*}(l) = \omega^2\widehat{(\epsp u^*)}(l)$, or
\begin{equation}\label{eq:Bxi_Fourier}
 \widehat{u^*}(l) = 
 \frac{\omega^2\widehat{(\epsp u^*)}(l)}{(\xi-\xi_l)(\xi-\overline{\xi_l})},
\qquad l=(m,n)\in\Z\times\N.
\end{equation}
This identity will be used extensively. 
Moreover, we need the lower bounds 
\begin{subequations}\label{eqs:aux_as_estim}
\begin{align}\label{eq:aux_as_estim1}
|\xi-\xi_l|^2&= (\Re \xi+2\pi \mind )^2+ (\Im \xi-\kappa_{\nind})^2
> \begin{cases}
(\pi \mind )^2, & \nind\in I_{\Scomp}\\
(\pi \mind )^2 + \delta_{\gamma}^2,&\nind\notin I_{\Scomp}
\end{cases}\\
\label{eq:aux_as_estim2}
|\xi-\overline{\xi_l}|^2&= (\Re \xi+2\pi \mind )^2 + (\Im \xi + \kappa_{\nind})^2
> (\pi \mind )^2+\paren{\kappa_{\Scomp} +\kappa_{\nind}}^2
\end{align}
\end{subequations}
which hold for all $\xi\in \Scomp$ and $l\in \Z\times\N\setminus 
\{0\}\times I_{\Scomp}$ and follow by separate estimation of real and imaginary
parts using \eqref{eq:use_of_delta_gamma}, $|\Re\xi|<\pi$, $\kappa_{\nind}\geq 0$
(see Fig.~\ref{subfig:FloquetModes}). 

Estimating $|(\xi-\overline{\xi_l})(\xi-\xi_l)|
\geq \min(\delta_{\gamma},\pi) \kappa_{\Scomp}$ for 
$l\in \Z\times\N\setminus \{0\}\times I_{\Scomp}$ we obtain
\begin{align*}
\lefteqn{\|u^*-P_{\Scomp}u^*\|^2_{L^2}=
\Big\|\sum_{l\in\Z\times\N\setminus \{0\}\times I_{\Scomp}} \widehat{u^*}(l) \varphi_l\Big\|_{L^2}^2
= \sum_{l\in\Z\times\N\setminus \{0\}\times I_{\Scomp}} 
\Big|\frac{\omega^2\widehat{(\epsp u^*)}(l)}
{(\xi-\overline{\xi_l})(\xi-\xi_l)}\Big|^2}\\
& \leq 
\paren{\frac{\omega^2}{\min(\delta_{\gamma},\pi)\kappa_{\Scomp}}}^2
\sum_{l\in\Z\times\N} \!\! |\widehat{(\epsp u^*)}(l)|^2
\leq \paren{\frac{\omega^2\bar\varepsilon}
{\min(\delta_{\gamma},\pi)\kappa_{\Scomp}}}^2 
\|u^*\|_{L^2}^2.
\end{align*}
To prove \eqref{eq:as_estim_u1} we use that 
$|\xi-\xi_l|\,|\xi-\overline{\xi_l}|\geq 
\kappa_{\Scomp}\pi |\mind |$ and 
$\partial_{x_1}\varphi_{(0,\nind)} =0$ to obtain
\begin{align*}
\|\partial_{x_1}u^*\|^2_{L^2} 
&= \Big\| \!\!\!\sum_{l\in (\Z\setminus\{0\})\times \N} \!\!\!\!\!\!\!\!\!
\widehat{u^*}(l)\partial_{x_1}\varphi_l\Big\|^2 
=\!\!\!\!\!\!\! \sum_{l\in (\Z\setminus\{0\})\times \N} \!\!\!
\Big|\frac{(2\pi \mind ) \omega^2 \widehat{(\epsp u^*)}(l)}
{(\xi-\overline{\xi_l})(\xi-\xi_l)}\Big|^2\\
&\leq \paren{\frac{2\omega^2\bar\varepsilon}
{\kappa_{\Scomp}}}^2\!\! \|u^*\|_{L^2}^2.
\end{align*}
Since 
\[
u^*(0,\cdot) = \sum_{l\in\Z\times\N} \widehat{u^*}(l)\varphi_l(0,\cdot)
= \sum_{\nind\in\N}\paren{\sum_{\mind \in\Z}\widehat{u^*}(\mind ,\nind)} \psi_{\nind},
\]
we have
\begin{equation}\label{eq:aux_trace}
\begin{split}
\|(u^*-P_{\Scomp}u^*)(0,\cdot)\|_{L^2((0,L))}^2
=& \sum_{\nind\in \N \setminus I_{\Scomp}} 
\Big|\sum_{\mind \in\Z}\widehat{u^*}(\mind ,\nind)\Big|^2\\
&+  \sum_{\nind\in  I_{\Scomp}} \Big|\sum_{\mind \in\Z\setminus\{0\}}\widehat{u^*}(\mind ,\nind)\Big|^2.
\end{split}
\end{equation}
Using the Cauchy-Schwarz inequality, \eqref{eq:Bxi_Fourier}, and
the lower bounds $|\xi-\xi_l|^2 \geq \pi^2\mind ^2+\delta_{\gamma}^2$ and
$|\xi-\overline{\xi_l}|^2\geq \kappa_{\Scomp}^2$ (see \eqref{eqs:aux_as_estim}), 
the first term in \eqref{eq:aux_trace} can be bounded by
\begin{align*}
\sum_{\nind\in \N \setminus I_{\Scomp}} 
\abs{\sum_{\mind \in\Z}\widehat{u^*}(\mind ,\nind)}^2
&\leq \paren{\sum_{\mind \in\Z} \frac{1}{\pi^2\mind ^2+\delta_{\gamma}^2}}
\sum_{\mind \in\Z,\nind\in\N\setminus I_{\Scomp}}
(\pi^2\mind ^2 + \delta_{\gamma}^2) |\widehat{u^*}(\mind ,\nind)|^2 \\
&\leq  C^2\sum_{\mind \in\Z,\nind\in\N\setminus I_{\Scomp}} 
\frac{\omega^4 |\widehat{(\epsp u^*)}(\mind ,\nind)|^2}{|\xi-\xi_l|^2} 
\leq C^2\frac{(\omega^2\bar\varepsilon)^2}{\kappa_{\Scomp}^{2}} \|u^*\|_{L^2}^2\,.
\end{align*}
Using \eqref{eq:as_estim_u1} the second term in \eqref{eq:aux_trace} can be estimated
by
\begin{align*}
\sum_{\nind\in  I_{\Scomp}} \Big|\sum_{\mind \in\Z\setminus\{0\}}\widehat{u^*}(\mind ,\nind)\Big|^2
&\leq  \paren{\sum_{\mind \in\Z\setminus\{0\}} \frac{1}{(2\pi \mind )^2}}
\paren{\sum_{\mind \in\Z\setminus\{0\},\nind\in I_{\Scomp}} 
\!\!\!\!\!\!\!\!\!
(2\pi \mind )^2 |\widehat{u^*}(\mind ,\nind)|^2}\\
&\leq \frac{C^2}{4} \|\partial_{x_1}u^*\|_{L^2}^2
\leq C^2 \frac{(\omega^2\bar\varepsilon)^2}{\kappa_{\Scomp}^2} \|u^*\|_{L^2}^2
\end{align*}
completing the proof of \eqref{eq:diri_trace}.

To prove \eqref{eq:neum_trace}, first note that
\[
\partial_{x_1}u^*(0,\cdot) = \sum_{l\in\Z\times\N} \widehat{u^*}(l)\partial_{x_1}\varphi_l(0,\cdot)
= \sum_{\nind\in\N}\paren{\sum_{\mind \in\Z\setminus\{0\}}(2\pi i \mind )\widehat{u^*}(\mind ,\nind)} \psi_{\nind}.
\]
It follows from \eqref{eq:Bxi_Fourier}, the Cauchy-Schwarz inequality, 
the estimate $|\xi-\xi_l|\,|\xi-\overline{\xi_l}|\geq (\pi \mind )^2$ 
(see \eqref{eqs:aux_as_estim}),
and the identity $\sum_{\mind =1}^{\infty} \mind ^{-2} = \frac{\pi^2}{6}$ that
\begin{align*}
\|\partial_{x_1}u^*(0,\cdot)\|^2_{L^2} =&
\sum_{\nind\in\N}
\abs{\sum_{\mind \in\Z\setminus\{0\}} \frac{\omega^2(2\pi i\mind )\widehat{(\varepsilon_p u^*)}(l)}
{(\xi-\xi_l)(\xi-\overline{\xi_l})}}^2\\
\leq&
\paren{2 \sum_{\mind =1}^{\infty}\frac{1}{\mind ^2}}
\sum_{l\in\Z\setminus\{0\}\times\N}\frac{(2\pi\omega^2 \mind ^2)^2 |\widehat{(\varepsilon_p u^*)}|^2}
{(\pi \mind )^4}
\leq \frac{4}{3} (\omega^2 \bar\varepsilon)^2  \|u^*\|_{L^2}^2,
\end{align*}
which proves \eqref{eq:neum_trace}.
\end{proof}

We can still fix a complex scaling constant in the Floquet modes
$\vp_n$ defined in \S \ref{sec:summary}. For our purposes it will be
sufficient to fix this constant for all but a finite number of these
$\vp_n$. We will assume in the following that 
\begin{equation}\label{eq:general_case}
\gamma\notin \{\gamma_0,\gamma_\pi\}\,.
\end{equation}
The special case $\gamma\in \{\gamma_0,\gamma_\pi\}$ will
be discussed in Appendix \ref{sec:appendix_quasiperiodic}. 
If \eqref{eq:general_case} holds true, it follows 
from the explicit form of $\kappa_n$ given in 
{\S} \ref{sec:preliminaries} that there exists $\tilde{N}\in\N$
such that for all $n>\tilde{N}$ the connected component 
$\Scomp_n$ of $\mathcal{S}$ containing
$\xiup_n$ is a disk containing precisely one characteristic value with
multiplicity. Therefore, the function 
\begin{equation}\label{eq:defi_un}
u_n(x):= e^{-i\xi_n x_1} \vp_n(x),\qquad n>\tilde{N}
\end{equation}
is periodic in $x_1$,  $P_{\Scomp_n}u_n$ is constant 
in the first variable, and $(P_{\Scomp_n}u_n)(0,\cdot)$ is a multiple
of $\psi_n$. Therefore, we can choose the free complex scaling constants of
$\vp$ and $u_n$ such that
\begin{equation}\label{eq:scaling_un}
(P_{\Scomp_n}u_n)(0,x_2) 
= \paren{\kappa_n+\frac{1}{2\kappa_n}}^{-1/2}\psi_n(x_2)\,.
\end{equation}
If $\omega=0$ this scaling yields $\vp_n=\vr_n$. 

\begin{corollary}\label{cor:estim_un}
There exist a constant $C>0$ depending only on  $\omega^2\bar\varepsilon$ 
and $\gamma$ such that for all $n>\tilde{N}$
\begin{subequations}\label{eqs:asym_estim_with_n}
\begin{align}
\label{eq:norm_un}
\kappa_n^{3/2} \Norm{u_n-P_{\Scomp_n}u_n}_{L^2(\Omega)} \leq &\, C,\\
\label{eq:estim_fn}
\kappa_n^{3/2} \Norm{(u_n-P_{\Scomp_n}u_n)(0,\cdot)}_{L^2((0,L))}\leq &\, C,\\
\label{eq:estim_uprime_n}
\kappa_n^{1/2} \Norm{\partial_{x_1}u_n|_{x_1=0}}_{L^2((0,L))} \leq &\, C.
\end{align}
\end{subequations}
\end{corollary}

\begin{proof}
It follows from  Lemma \ref{lemm:FloquetModes} that 
$\|u_n\|_{L^2(\Omega)}\leq 2\|P_{\Scomp_n}u_n\|_{L^2(\Omega)}$
for sufficiently large $n$. Moreover, due to \eqref{eq:scaling_un}
we have $\|P_{\Scomp_n}u_n\|_{L^2(\Omega)}= 
\paren{\kappa_n+\frac{1}{2\kappa_n}}^{-1/2}$. Now the assertions
follow from Lemma \ref{lemm:FloquetModes}. 
%
%
\end{proof}

\subsection{Properties of the operator $T$}

\begin{proposition}[properties of $T$]\label{prop:T_well_defined}
\begin{enumerate}
\item
The operator $T$
is well-defined and bounded, and $T-\Tref$ is compact. 
\item
$T(a_n)$ satisfies \eqref{eq:pde} and \eqref{eq:bc_top_bottom} for all 
$(a_n)\in l^2(\N)$.
\end{enumerate}
\end{proposition}

\begin{proof}
\emph{Part 1:} In the following $C$ denotes a generic constant 
depending only on $\omega^2\bar\varepsilon$ and $\gamma$.
Note that $\vr_n(x)= \exp(-\kappa_nx_1)\tilde{\psi}_n(x_2)$ with
$\tilde{\psi}_n(x_2):= (P_{\Scomp_n}u_n)(0,x_2)$ for $n>\tilde{N}$.
Inserting the term $\pm e^{i\xi_nx_1}\tilde{\psi}_n(x_2)$ 
in $\|\vp_n-\vr_m\|_{L^2}$ 
and applying the triangle inequality yields the estimate
\begin{align} \label{WL2}
\|\vp_n-\vr_n\|_{L^2(S^+)} 
&\leq
\Norm{e^{i\xi_nx_1}\paren{u_n-P_{\Scomp_n}u_n}}_{L^2(S^+)}
+ \Norm{e^{i\xi_nx_1}-e^{-\kappa_nx_1}}_{L^2}\|\tilde{\psi}_n\|_{L^2} \nonumber \\
&\leq
\frac{\Norm{u_n-P_{\Scomp_n}u_n}_{L^2(\Omega)}}{1-e^{-\Im \xi_{\tilde{N}}}}
+ \paren{\int_0^{+\infty}\abs{e^{i\xi_n x_1}-e^{-\kappa_nx_1}}^2}^{\frac{1}{2}}
\|\tilde{\psi}_n\|_{L^2} \nonumber \\
&\leq \frac{C}{\kappa_n^{3/2}},\qquad n>\tilde{N}.
\end{align}
In the second line we have used that $0<\Im\xi_{\tilde{N}}\leq\Im\xi_n$ for all $n> \tilde{N}$ and
$\sum_{k=0}^\infty e^{-k\Im\xi_n} = (1-e^{-\Im\xi_n})^{-1}
\leq (1-e^{-\Im\xi_{\tilde{N}}})^{-1}$, 
and in the third line \eqref{eq:norm_un} was applied together with the
identity
$$\int_0^{+\infty} \abs{e^{i\xi_n x_1}-e^{-\kappa_nx_1}}^2 \,dx_1=
\frac{\xi_n-i\kappa_n}{2}
\paren{
\frac{1}{\Im\xi_n(\bar\xi_n-i\kappa_n)}+\frac{1}{\kappa_n(\xi_n+i\kappa_n)}},$$
and Proposition \ref{prop:large_char_val}.

To obtain an identity for the $L^2$-distance of the gradients, 
we apply Green's first theorem in $(0,l)\times (0,L)$, use the identity 
$\Delta\paren{\vp_n-\vr_n} =-\omega^2\epsp \vp_n$, and let $l\to\infty$:
\begin{align*}
\Norm{\nabla(\vp_n-\vr_n)}_{L^2(S^+)}^2
=& \omega^2\int_{S^+}\epsp \vp_n (\overline{\vp_n}-\overline{\vr_n})dx \\
&-\int_0^L\diffq{(\vp_n-\vr_n)}{x_1}(0,x_2)(\overline{\vp_n}-\overline{\vr_n})(0,x_2)\,dx_2
\end{align*}
Since $\diffq{(\vp_n-\vr_n)}{x_1}(0,x_2)=i\xi_n u_n(0,x_2)+\kappa_n\tilde{\psi}_n
+\diffq{u_n}{x_1}(0,x_2)$, it follows after adding $\pm i\xi_n \tilde{\psi}_n$ and
using \eqref{eq:estim_fn}, \eqref{eq:estim_uprime_n}, and Proposition  \ref{prop:large_char_val} that 
\linebreak
$\|\diffq{(\vp_n-\vr_n)}{x_1}(0,\cdot)\|_{L^2}\leq C/\sqrt{\kappa_n}$.
Using Cauchy's inequality, \eqref{WL2}, and \eqref{eq:estim_fn} yields 
\begin{equation}\label{Wgrad}
\Norm{\nabla (\vp_n-\vr_n)}_{L^2(S^+)}^2 \leq \frac{C}{\kappa_n^2},\qquad 
n>\tilde{N}.
\end{equation}
Define $K_j:V\to H^{1,+}_{\gamma}(S^+)$ by $K_j(a_n):= \sum_{n=1}^ja_n(\vp_n-\vr_n)$.
Combining \eqref{WL2} and \eqref{Wgrad} and using Cauchy's
inequality, we deduce that
\[
\|(K_{m_2}-K_{m_1})(a_n)\|_{H^1(S^+)}^2
\leq C\paren{\sum_{n=m_1+1}^{m_2} \frac{1}{\kappa_n^3} 
+ \sum_{n=m_1+1}^{m_2}\frac{1}{\kappa_n^2}} \|(a_n)\|^2,
\]
which implies together with \eqref{eq:sigma_ls} 
that $(K_j)$ is a Cauchy sequence with respect to 
the operator norm. Therefore,  $K=\lim_{j\to\infty}K_j$
is well defined, and since the range of the operators $K_j$ is 
finite dimensional, $K$ is compact. Moreover, $T=\Tref+K$ is well-defined and bounded.

\emph{Part 2:} 
Since the differential operator $\Delta + \omega^2\epsp$ is continuous from 
$H^1_{\gamma}(S^+)$ to $H^{-1}_{\gamma}(S^+)$, we can interchange its application
with summation to show that $w:=T(a_n)$ satisfies \eqref{eq:pde}. Analogously, it
follows from the continuity of the trace operators that $w$ satisfies
\eqref{eq:bc_top_bottom}.
\end{proof}

\section{Proof of the Main Theorem}\label{sec:main_thm_proof}
\begin{proposition}\label{prop:T_injective}
The operator $T$ is injective.
\end{proposition}

\begin{proof}
Let $v:= \sum_{n=1}^{\infty}a_n \vp_n$ for some sequence $(a_n)\in l^2(\N)$
and assume that $v\equiv 0$. We have to show that $a_n=0$ for all $n\in\N$.
Let $\{\nu_1,\nu_2,\dots\} = \{\Im \xi_n:n\in\N\}$ with $0\leq \nu_1<\nu_2<\cdots$.
We show by induction in $l\in\N$ that $a_n=0$ for all $n\in\N$ satisfying 
$\Im \xi_n\leq \nu_l$: Formally adding $\nu_0:=-1$, the induction base is trivial.
For the induction step assume that the statement holds true for $l-1$ and
that there are precisely $M$ distinct characteristic values $\tilde{\xi}_1,
\dots \tilde{\xi}_M$ with $\Im \tilde{\xi}_m=\nu_l$ and $\Re \tilde{\xi}_m\in
[-\pi,\pi)$. Pecularities for the case $l=1$ and $\nu_l=0$ will be discussed at
the end of the proof, so assume for the moment that $\nu_l>0$. 
Suppose that $\mathfrak{n}((B_{\xi}),\tilde{\xi}_m) = 
(r_1^{(m)},\dots,r_{\alpha_m}^{(m)})$. Relabelling the Floquet modes $\vp_n$
and the coefficients $a_n$ corresponding to $\tilde{\xi}_1,\dots,\tilde{\xi}_M$ 
by $v_{j,k}^{(m)}$ and $a_{j,k}^{(m)}$ we obtain from the induction hypothesis that for any $\epsilon>0$ and all $x_2\in [0,L]$
\begin{equation}\label{eq:aux_injective}
\exp(\nu_l x_1) \sum_{m=1}^{M} \sum_{j=1}^{\alpha_m} \sum_{k=0}^{r_j^{(m)}-1}
a_{j,k}^{(m)}  v_{j,k}^{(m)}(x) 
= O(e^{-(\nu_{l+1}-\nu_l-\epsilon)x_1}),\qquad x_1\to\infty
\end{equation}
and have to show that all coefficients $a_{j,k}^{(m)}$ vanish. For simplicity,
let us first assume that all partial null multiplicities $r_j^{(m)}$ are 
equal to $1$. With the notation of Proposition \ref{prop:char_val_eig} 
(and an additional index $m$) the function 
$\exp(-i\tilde{\xi}_m x_1) \sum_{j=1}^{\alpha_m} a_{j,0}^{(m)} v_{j,0}^{(m)}(x)
= \sum_{j=1}^{\alpha_m} a_{j,0}^{(m)} \tilde{u}^{(j)}_{0,m}(x)$
is $1$-periodic in $x_1$ for each $m$. Pick some $x=(x_1,x_2)\in\Omega$ and set
$g_m:= \sum_{j=1}^{\alpha_m} a_{j,0}^{(m)} \tilde{u}^{(j)}_{0,m}(x)$.
Then it follows from \eqref{eq:aux_injective} that 
\begin{align}\label{eq:aux_uniqueness}
&\sum_{m=1}^M\exp(i\Re \tilde{\xi}_m)^p g_m 
= \exp(\nu_l p) \sum_{m=1}^M \exp(i\tilde{\xi}_m p) 
\sum_{j=1}^{\alpha_m} a_{j,0}^{(m)} \tilde{u}^{(j)}_{0,m}(x)\\
&= \exp(\nu_l p) \sum_{m=1}^M \sum_{j=1}^{\alpha_m} a_{j,0}^{(m)} 
v^{(j)}_{0,m}(x_1+p,x_2) \to 0,\qquad \mbox{as } p\to\infty\,,p\in\N\,.\nonumber
\end{align}
If the right hand side would vanish exactly for $M$ consecutive values
of $p$, say $p\in \{q+1,\dots,q+M\}$, we could  conclude immediately that
$g:=(g_1,\dots,g_M)^{\top}\in\C^M$ 
is zero since the matrix $A^{(q)}\in\C^{M\times M}$
defined by $A^{(q)}_{lm} := \exp(i\Re \tilde{\xi}_{m})^{l+q}$, $l,m=1,\dots,M$
is regular as shown below. However, \eqref{eq:aux_uniqueness} only implies that
$\lim_{q\to\infty} A^{(q)}g= 0$. Therefore, we have to control 
$\|[A^{(q)}]^{-1}\|$ uniformly in $q$. For this end note that $A^{(q)}$ 
has a factorization
\[
A^{(q)} = A^{(0)}
 \diag\left(\exp(i\Re \tilde{\xi}_{1})^{q},\dots,\exp(i\Re \tilde{\xi}_{M})^{q}\right)
\]
and $A^{(0)}$ is a Vandermonde matrix. It follows that $A^{(q)}$ is 
invertible and $\|[A^{(q)}]^{-1}\|$
is independent of $q$. Therefore, $\lim_{q\to\infty} A^{(q)}g= 0$
implies $g=0$, 
i.e.~$\sum_{j=1}^{\alpha_m} a_{j,0}^{(m)} \tilde{u}^{(j)}_{0,m}\equiv 0$ 
for each $m=1,\dots,M$.
Since we know from Proposition \ref{prop:char_val_eig} that
$\tilde{u}^{(1)}_{0,m},\dots,\tilde{u}^{(\alpha_m)}_{0,m}$ are linearly independent, 
it follows that $a_{1,0}^{(m)}=\dots = a_{\alpha_m,0}^{(m)}=0$.

Now assume that $\overline{r}:=\max\{r_{j}^{(m)}:m=1,\dots,M,j=1,\dots,\alpha^{(m)}\}$
is greater than $1$. Multiplying \eqref{eq:aux_injective} by $x_1^{-\overline{r}+1}$
and using \eqref{eq:formFloquetBasis} 
it follows that \eqref{eq:aux_uniqueness} holds with
$g_m:= \sum_{j:r^{(m)}_j=\overline{r}} a_{j,\overline{r}-1}^{(m)} \tilde{u}^{(j)}_{0,m}(x)$ (with the convention that an empty sum is $0$),
and it follows as above that $g_1=\cdots = g_M=0$ for all $x\in\Omega$. 
Using again the linear indepence of the functions $\tilde{u}^{(j)}_{0,m}$,
we find that $a_{j,\overline{r}-1}^{(m)}=0$ for all $(j,m)$ such that
$r_j^{(m)}=\overline{r}$. In a second step we show analogously 
that $a_{j,\overline{r}-2}^{(m)}=0$ for all $(j,m)$ such that
$r_j^{(m)}\geq \overline{r}-1$, and so on. 
Finally, all coefficients $a_{j,k}^{(m)}$ have to vanish. 

It remains to discuss the case of propagating modes, i.e.\ the induction step
for $\nu_1=0$ since not all elements of the eigenspaces are considered here.
However, we can use the same technique as above to prove the stronger result
that  $\sum_{n=1}^{\overline{n}}a_n^- v_n^- + \sum_{n=1}^{\infty} a_n^+\vp_n\equiv 0$ 
implies $0=a_1^-=\dots = a_{\overline{n}}^-=a_1^+=a_2^+=\cdots$. 
\end{proof}

The well posedness result in the following proposition follows from general
results in \cite{B:nazarov_plamenevsky}, but we obtain an independent proof
as a side product of our analysis:
\begin{proposition}[well-posedness of problem \eqref{dirprob}]\label{prop:wellposed}
Assume that the only solution $v\in H^{1,+}_{\gamma}(S^+)$ 
to \eqref{dirprob} with $f=0$ is $v=0$. 
Then $F:=\ltrace T$ is bounded and boundedly invertible, and 
problem \eqref{dirprob} is well posed in the sense that for
all $f\in \lbdspace$ there exists a unique solution 
$v\in H^{1,+}_{\gamma}(S^+)$ to \eqref{dirprob}, and $v$ 
depends continuously on $f$. 
\end{proposition}

\begin{proof}
$F$ is a Fredholm operator with index $0$ since
$F=\ltrace \Tref+ \ltrace (T-\Tref)$, and $\ltrace \Tref$ is Fredholm with index $0$ 
by Lemma \ref{lemm:T0}, and $\ltrace (T-\Tref)$ is compact by
Proposition \ref{prop:T_well_defined}. Assume that $F(a_n)=0$ and set $v:=T(a_n)$. 
Then $v$ is a solution to \eqref{dirprob} with $f=0$, and hence $v=0$ by our assumption. 
Using Proposition \ref{prop:T_injective} we conclude that $(a_n)=0$,
i.e.\ $F$ is injective. This implies that $F$ has a bounded inverse. 
Using Proposition \ref{prop:T_well_defined}, it follows that $v:= T F^{-1} f$
is a solution of \eqref{dirprob}, which depends continuously on $f$.
\end{proof}

The proof of our main theorem is now simple:

\begin{proof}[Proof of Theorem \ref{theo:main}]
We start with part 2 of the theorem: Since $F:=\ltrace T$ is bounded and boundedly invertible 
by Proposition \ref{prop:wellposed}, the set $\{\ltrace \vp_n:n\in\N\}$ is a Riesz basis of 
$\lbdspace$. 
It follows from Propositions \ref{prop:wellposed} and \ref{prop:uniqueness_robin}
that \eqref{dirprob} is well posed for 
$\ltrace v = v(0,\cdot)+i\diffq{v}{x_1}(0,\cdot)$.
For this choice of $\ltrace$, the operator $T$ has the bounded left inverse $F^{-1}\ltrace$,
and hence $\{\vp_n:n\in\N\}$ is a Riesz basis of $\mathrm{ran}(T)=V$. 

Since the functions $\vp_n$ are chosen as in Proposition
\ref{prop:char_val_eig}, the matrix representing $\mathcal{T}$ consists
of Jordan blocks. 
Because of Proposition \ref{prop:large_char_val} at most a finite number of
these Jordan blocks have size $>1$. It is straightforward to see that $\mathcal{R}$
is represented by the same matrix. 
\end{proof}


\appendix
\section{On the radiation condition}\label{appendix:radiation}
\renewcommand{\thesection}{A}

The study of solutions to the Helmholtz equation $\Delta v + \omega^2 \epsp v =0$ in $S$
with boundary conditions $\gamma v(x_1,\cdot) = \zerotwo$ for $x_1\in\R$ amounts to
the study of spectral properties of the operator
$A^{(\gamma)}:=- \frac{1}{\epsp}\Delta: H^2_{\gamma}(S)\to L^2(S)$.
Due to the isometry of the Floquet transform, the spectrum of $A_{\gamma}$ is
the union of the spectra of the operators defined by
\[
A_\alpha^{(\gamma)}:=-\frac{1}{\epsp}\Delta_\alpha:\quad  
H^2_{\gamma}(\Omega)\to L^2(\Omega),
\qquad \mbox{for }\alpha\in \R.
\]
Since these operators are positive and self-adjoint in the weighted 
Hilbert space $L^2(\Omega,\epsp)$ with a compact 
resolvent, their spectra consist of a countable number of
positive eigenvalues with finite multiplicities accumulating only at $\infty$:
$$
\sigma(A_\alpha^{(\gamma)})=\{\tilde{\lambda}_m(\alpha):m\in\N\},\qquad \alpha\in [-\pi,\pi).
$$
We assume that the $\tilde{\lambda}_m$ are arranged in increasing order.
The functions $\tilde{\lambda}_m$ are smooth, except at points
where two or more of them cross. Alternatively, the eigenvalues of $A_{\alpha}^{\gamma}$
can be arranged such that they are analytic functions $\lambda_m$ of $\alpha$, 
which have holomorphic extensions to a complex neighborhood $\mathcal{U}$ of 
$[-\pi,\pi)$ (\cite{SW:02}).  
Furthermore, there exists a holomorphic family of eigenfunctions 
$\mathcal{U}\to L^2(\Omega)$, 
$\xi\mapsto w_{n,\xi}$:
\begin{equation}\label{eq:eigenvec_alpha}
\Delta_{\xi} w_{m,\xi} + \lambda_m(\xi) \epsp w_{m,\xi} =0,
\qquad \xi\in \mathcal{U}, m\in\N.
\end{equation}

Note that if $\xi^*\in [-\pi,\pi)$ is a 
characteristic value of $(B_\xi)$, then 
\begin{equation}\label{eq:kernel_Bxi}
\ker B_{\xi^*} = \ker(A^{(\gamma)}_{\xi^*}-\omega^2I) 
= \mathrm{span}\{w_{m,\xi^*}:\exists m\in\N \,\,\lambda_m(\xi^*)=\omega^2\}.
\end{equation}
\begin{proposition}\label{prop:degenerate_char_val}
$\xi^*\in [-\pi,\pi)$ is a characteristic value of
$(B_\xi)$ with a partial null multiplicity greater than 1 if and only if there
exists $m\in\N$ such that $\lambda_m(\xi^*)=\omega^2$ and $\lambda_m'(\xi^*)=0$. 
%
\end{proposition}

\begin{proof} 
Suppose that $\omega^2=\lambda_m(\xi^*)$ and $\lambda_m'(\xi^*)=0$. 
Taking the derivative of \eqref{eq:eigenvec_alpha} with respect to $\xi$,
which will be indicated by a prime in the rest of this proof, yields
\begin{equation}\label{eq:wmprime}
\left(\Delta_{\xi}+\lambda_m(\xi)\epsp\right)w'_{m,\xi}
+2i(\partial_{x_1}+i\xi)w_{m,\xi}+\lambda_m'(\xi) \epsp w_{m,\xi}=0.
\end{equation}
Since $\lambda_m'(\xi^*)=0$, we have
\[
\frac{dB_\xi w_{m,\xi}}{d\xi}|_{\xi=\xi^*}=0
\qquad \mbox{and}\qquad 
B_{\xi^*}w_{m,\xi^*}=0
\]
Therefore,  $(w_{m,\xi})$ is a root function of $(B_\xi)$ corresponding to
$\xi^*$ with a partial null multiplicity greater than $1$.

Conversely, assume that $\xi^*$ is a characteristic value of $B_\xi$
with a partial null multiplicity greater than 1. Then, there exists a root function
$(u_\xi)$ such that $B_{\xi^*} u_{\xi^*}=0$
and
\begin{equation}\label{EqBuuprime}
B_{\xi^*} u'_{\xi^*}+B'_{\xi^*}u_{\xi^*}=0.
\end{equation}
Let $\Xi(\omega^2,\xi^*):=\{m\in\N:\lambda_m(\xi^*)=\omega^2\}$.
Due to \eqref{eq:kernel_Bxi} there exist coefficients $\nu_m\in \C$ such that
$w_{\xi^*} = u_{\xi^*}$ with $w_{\xi}:=\sum_{m\in\Xi(\omega^2,\xi^*)}\nu_m w_{m,\xi}$. 
Taking a linear combination of the equations \eqref{eq:wmprime} and
subtracting \eqref{EqBuuprime}, we get
\begin{equation}\label{eq:aux_Jordanreal}
B_{\xi^*}\paren{w'_{\xi^*}-u'_{\xi^*}}
= - \sum_{m\in \Xi(\omega^2,\xi^*)} \lambda_m'(\xi^*) \nu_m \epsp w_{m,\xi^*}
\end{equation}
The right hand side of \eqref{eq:aux_Jordanreal} belongs to
$\mathrm{ran}(B_{\xi^*})$. 
Since $\xi^*\in\R$, $B_{\xi^*}$ is self-adjoint in $L^2(\Omega)$, 
and hence $\ker(B_{\xi^*}) = \mathrm{ran}(B_{\xi^*})^{\perp}$. 
As $$\sum_{m\in \Xi(\omega^2,\xi^*)} \lambda_m'(\xi^*) \nu_m w_{m,\xi^*}\in\ker(B_{\xi^*}),$$ we have
\[\int_\Omega \epsp \abs{\sum_{m\in \Xi(\omega^2,\xi^*)} \lambda_m'(\xi^*) \nu_m w_{m,\xi^*}}^2\,dx =0.
\]
As the functions $w_{m,\xi^*}$ are linearly independent, it follows 
that $\lambda_m'(\xi^*) \nu_m=0$ for all $m\in \Xi(\omega^2,\xi^*)$. 
Since not all $\nu_m$ vanish, we obtain that $\lambda_m'(\xi^*)=0$ for some $m$.
\end{proof}

Recall that the group velocity of a Floquet mode $w_{m,\xi}(x)e^{i\xi x_1}$
is given by
\begin{equation}
\frac{d\omega}{d\xi} = \frac{d \sqrt{\lambda_m(\xi)}}{d\xi} 
= \frac{\lambda_m'(\xi)}{2\sqrt{\lambda_m(\xi)}}.
\end{equation}

\begin{proposition}\label{prop:radiation}
For a Floquet mode of the form $v(x)=w_{m,\xi^*}(x)
\linebreak
e^{i\xi^* x_1}$,
$\lambda_m(\xi^*)= \omega^2$ with
non vanishing group velocity, i.e. $\lambda_m'(\xi^*)\neq 0$ the
following statements are equivalent:
\begin{enumerate}
\item\label{it:energy_flux}
$v$ has positive energy flux, i.e.~$\Im q(v,v)>0$.
\item\label{it:group_vel}
$v$ has positive group velocity, i.e.\ $\lambda_m'(\xi^*)>0$.
\end{enumerate}
Moreover, if $\tilde{v}(x)=w_{n,\xi^*}(x)e^{i\xi^* x_1}$ is another Floquet mode with 
$\lambda_n(\xi^*)=\omega^2$ and $m\neq n$, then
$q(v,\tilde{v})=0$. 
\end{proposition}

\begin{proof}
Since $q_{x_1}(v,\tilde{v})=q(v,\tilde{v})$ is independent of $x_1$, it follows that 
\begin{align*}
q(v,\tilde v)=&
\int_{0}^L \bracket{\frac{\partial v}{\partial x_1}(x) \overline{\tilde{v}(x)}
- v(x)\overline{\frac{\partial \tilde{v}}{\partial x_1}(x)}}\, dx_2 \\
=& \int_0^1\int_0^L \bracket{\frac{\partial v}{\partial x_1}(x) \overline{\tilde{v}}(x)
- v(x)\overline{\frac{\partial \tilde{v}}{\partial x_1}(x)}}\, dx_2\,dx_1 \\
=& \int_\Omega \bracket{(\partial_{x_1}+i\xi^*)w_{m,\xi^*}(x) \overline{w_{n,\xi^*}(x)} 
-w_{m,\xi^*}(x) \overline{(\partial_{x_1}+i\xi^*)w_{n,\xi^*}(x)} }\, dx .
\end{align*}
Taking the $L^2$ inner product with $w_{n,\xi^*}$ in \eqref{eq:wmprime} and using 
$\lsp B_{\xi^*}w'_{m,\xi^*},w_{n,\xi^*}\rsp = \lsp w'_{m,\xi^*}, B_{\xi^*}w_{n,\xi^*}\rsp =0$
and analogously with the roles of $w_{m,\xi^*}$ and $w_{n,\xi^*}$ interchanged 
we obtain that 
\[
q(v,\tilde{v})=\frac{i}{2}(\lambda'_m(\xi^*) + \lambda_n'(\xi^*))
\int_{\Omega} \epsp w_{m,\xi^*} \overline{w_{n,\xi^*}}\, dx.
\]
Since $w_{m,\xi}$ and $w_{n,\xi}$ are orthogonal with respect to the inner product
in
\linebreak
$L^2(\Omega,\epsp)$ for all $\xi\in\R$ as eigenfunctions of the self-adjoint operators
$A_{\xi}^{\gamma}$ in this space, we obtain the last statement. 
Choosing $\tilde{v}=v$ shows the equivalence result.
\end{proof}

\begin{remark}
Two further important items could be added to the list of equivalent statements
in Proposition \ref{prop:radiation}, which we do not want to define in detail here: 
It has been shown by Fliss \cite[Theorem 3.2.57]{Fliss:09} that a Floquet mode
has positive group velocity if and only if it satisfies the limit absorption principle. Moreover, the equivalence of  the principles of limit absorption and limit amplitude
has been shown in the thesis of Radosz \cite{radosz:10}. 
These are very strong indications that solutions with positive group velocity are 
``physical solutions'', and Proposition \ref{prop:radiation} shows that
for frequencies $\omega$ for which all real characteristic values have total multiplicity 1
(cf.\ Proposition \ref{prop:degenerate_char_val} for a characterization), 
the conditions \eqref{eq:orthonormalq} are satisfied precisely for the 
physical Floquet modes (up to normalization).

However, if two (or more) bands $\lambda_m$ and $\lambda_n$ cross at $\omega^2$,
i.e.\ $\lambda_m(\xi^*)=\lambda_n(\xi^*)=\omega^2$ for some $\xi^*\in [-\pi,\pi)$ and if 
$\lambda_m'(\xi^*)\lambda_n'(\xi^*)<0$, then a system of Floquet modes
satisfying \eqref{eq:orthonormalq}
does not necessarily satisfy the limiting absorption principle.
E.g., the Floquet modes $v_a:=\sqrt{2}v+\tilde{v}$,
$v_b:=\sqrt{2}\tilde{v}+v$ satisfy \eqref{eq:orthonormalq}, but $v_a$ does
not satisfy the limiting absorption principle. 
This shows in particular that Conjecture 4.2 
in \cite{ESZ:09} is false in general,  but true if all real characteristic values have 
total multiplicity 1. 

For Floquet modes with group velocity $0$ we do not know which are
the physical solutions. A system of Floquet modes satisfying \eqref{eq:orthonormalq}
is constructed in \cite{B:nazarov_plamenevsky}, but we are not aware of indications
that they correspond to physical solutions. 
\end{remark}

%

\renewcommand{\thesection}{Appendix B}
\section{Uniqueness results}\label{appendix:uniqueness}
\renewcommand{\thesection}{B}
In this appendix we discuss conditions under which solutions to the boundary value
problem \eqref{dirprob} are unique. In this case problem \eqref{dirprob} is well posed 
(Proposition \ref{prop:wellposed}), and the second part of
our main theorem \ref{theo:main} holds true. Proposition \ref{prop:uniqueness_robin}
is used in the proof of Theorem \ref{theo:main}. 

\begin{lemma}\label{lemm:reduction_toH1}
Assume that either $\theta_{\rm N}=0$ (Dirichlet condition) or
$\theta_{\rm N}=1$ and 
$\Im \theta_{\rm D}\geq 0$.
Moreover, assume that the only solution 
$v\in H^1_{\gamma}(S^+)$ to the boundary value problem 
\eqref{eq:pde}-\eqref{eq:bc_left} with $f=0$ is $v=0$. 
Then the only solution  $v\in H^{1,+}_{\gamma}(S^+)$ to
\eqref{eq:pde}-\eqref{eq:bc_left} with $f=0$ is $v=0$. 
\end{lemma}

\begin{proof}
Let $v = \sum_{n=1}^{\overline{n}} c_n\vp_n +w\in H^{1,+}_{\gamma}(S^+)$ be 
a solution to \eqref{eq:pde}-\eqref{eq:bc_left} with $f=0$. Then it follows
from \cite[Theorem 5.1.4]{B:nazarov_plamenevsky} that there exists $\delta>0$ such that
$\exp(\delta x_1)w\in H^1_{\gamma}(S^+)$. Therefore, letting $x_1$ tend to $\infty$
in the definition of $q(w,v_n^{+})$, it follows that $q(w,\vp_n)=0$ for all $j$,
and analogously $q(w,w)=0$.
For the case $\theta_{\rm N}=1$ and 
$\Im \theta_{\rm D}\geq 0$ we obtain from $\diffq{v}{x_1}(0,x_2)=-\theta_{\rm D}v(0,x_2)$
and the orthogonality and normalization conditions
for the Floquet modes $\vp_n$ that
\begin{align*}
0 \geq -2(\Im \theta_{\rm D}) \int_0^L|v(0,x_2)|^2\, dx_2 
= 2\Im \int_0^L \diffq{v}{x_1}(0,x_2)\overline{v(0,x_2)}\, dx_2 
&= \Im q(v,v)\\
&= \sum_{j=1}^{\overline{n}} |c_j|^2. 
\end{align*}
This implies $c_1=\cdots = c_{\overline{n}}=0$ and hence $v=w$. 
For the case $\theta_{\rm N}=0$ it can be shown analogously that $v=w$. 
Now the assumption of the lemma implies $v=0$. 
\end{proof}

\begin{proposition}\label{prop:uniqueness_robin}
Let $\ltrace v = \diffq{v}{x_1}(0,\cdot)+i\kappa v(0,\cdot)$ with $\kappa>0$.
Then the only solution $v\in H^{1,+}_{\gamma}(S^+)$ to \eqref{dirprob} with $f=0$
is $v=0$. 
\end{proposition}

\begin{proof}
According to Lemma \ref{lemm:reduction_toH1} it suffices to consider solutions $v$
to \eqref{dirprob} with $f=0$ in $H^1_{\gamma}(S^+)$. 
Then there exists a sequence $R_k$ tending to $\infty$ such that 
$\lim_{k\to\infty} |\int_0^L \diffq{v}{x_1}(R_k,x_2)\overline{v(R_k,x_2)}\,dx_2| = 0$, and
hence $q(v,v)=0$. Therefore,
\[
0 = \Im q(v,v) = 2 \Im \int_0^L \diffq{v}{x_1}(0,x_2)\overline{v}(0,x_2)\,dx_2
= 2\kappa \int_0^L |v(0,x_2)|^2\,dx_2.
\]
Hence, $v(0,\cdot)=0$, and $\diffq{v}{x_1}(0,\cdot)= i\kappa v(0,\cdot)=0$.
Now a unique continuation principle in two dimensions (see 
\cite[Corollary 7.4.2]{schulz:90}) implies that $v\equiv 0$. 
\end{proof}

\begin{proposition}\label{prop:uniqueness_diri_neum}
Assume that $\ltrace$ is either the Dirichlet or the Neumann trace and that
$\epsp$ satisfies the symmetry condition $\epsp(1-x_1,x_2)=\epsp(x_1,x_2)$.
Then the only solution $v\in H^{1,+}_{\gamma}(S^+)$ to \eqref{dirprob} with $f=0$
is $v=0$. 
\end{proposition}

\begin{proof}
In both cases Lemma \ref{lemm:reduction_toH1} applies, so it suffices to show that 
a solution $v\in H^1_{\gamma}(S^+)$ to \eqref{dirprob} with $f=0$ vanishes. 
For the Dirichlet condition $v(0,\cdot)=0$ (i.e.\ $\theta_{\rm D}=1$ and $\theta_{\rm N}=0$)
we consider the odd extension $v(-x_1,x_2):=-v(x_1,x_2)$ for $x_1>0$, $x_2\in (0,L)$, and
for the Neumann condition $\diffq{v}{x_1}(0,\cdot)=0$ 
(i.e.\ $\theta_{\rm N}=1$ and $\theta_{\rm D}=0$) the even extension 
$v(-x_1,x_2):=v(x_1,x_2)$. In both cases the extended function $v$ satisfies the differential
equation \eqref{eq:pde} 
in the full strip $S$, and hence $B_{\alpha}(\mathcal{F}v)(\cdot,\alpha)=0$ for all $\alpha$.
Since $(B_{\xi})$ has at most a finite number of characteristic values on the real
axis, it follows that $(\mathcal{F}v)(\cdot,\alpha)=0$ for almost all $\alpha$. As
$v\in L^2(S)$, an application of the inverse Floquet transform yields $v\equiv 0$. 
\end{proof}

\renewcommand{\thesection}{Appendix C}
\section{Quasi-periodic boundary conditions with $\beta\in\{0,\pi\}$}\label{sec:appendix_quasiperiodic}
If $\gamma\in \{\gamma_0,\gamma_\pi\}$, 
there are infinitely many connected components of $\mathcal{S}$ containing
two characteristic values (with multiplicities). Here we set 
$\tilde{N}:=N$. Then each characteristic value $\xi_n^+$ with $n>\tilde{N}$
shares a 
connected component $\Scomp_n$ of $\mathcal{S}$ with precisely one other
characteristic value, w.l.o.g.\ $\xi_{n+1}^+$.  
For $\gamma\in\{\gamma_0,\gamma_\pi\}$ we cannot exclude the possibility 
of infinitely many characteristic values with partial null multiplicity
$2$ in general. Treating this case would require an extension of 
Lemma \ref{lemm:FloquetModes} and considerable additional work. 
To get along with Lemma \ref{lemm:FloquetModes} here, we have imposed
the additional assumption $\epsp(x_1,x_2) = \epsp(x_1,L-x_2)$ in
the case $\gamma\in\{\gamma_0,\gamma_\pi\}$ above. 
Then by a symmetry argument we can split the problem into
two problems for wave guides of width $L/2$: For $\beta=0$ we either
impose Dirichlet conditions  at $\{x:x_2\in\{0,L/2,L\}\}$ or Neumann
conditions at $\{x:x_2\in\{0,L/2,L\}\}$, and for $\beta=\pi$ we
impose Dirichlet conditions at $x_2=0$, $x_2=L$ and Neumann conditions
at $x_2=L/2$, or Neumann conditions at $x_2=0$, $x_2=L$ and Dirichlet conditions
at $x_2=L/2$. For each of the subproblems for waveguides of width $L/2$
we can exclude that possibility of infinitely many characteristic
values with partial null multiplicities $\geq 2$ as above.
Therefore, $u_n$ defined by \eqref{eq:defi_un} is again periodic in $x_1$.
However,  $P_{\Scomp_n}u_n$ is not necessarily a multiple
of $\psi_n$, but we only have 
$P_{\Scomp_n}u_n\in\mathrm{span}\{\psi_n,\psi_{n\pm 1}\}$.
Therefore, we replace \eqref{eq:scaling_un} by the normalization condition
\begin{align}\label{eq:normalization_un}
&\|P_{\Scomp_n}u_n\|_{L^2(\Omega)} =
\left(\kappa_n + \frac{1}{2\kappa_n}\right)^{-1/2},\qquad n>\tilde{N}\,,
\end{align}
which leaves a free phase factor. Moreover, we replace \eqref{eq:defi_vr} 
by
\[
\vr_n(x):=e^{-\kappa_n x_1} (P_{\Scomp_n}u_n)(x),\qquad n>\tilde{N}\,.
\]
It is easy to see that Lemma \ref{lemm:T0} and Corollary 
\ref{cor:estim_un} remain valid. The proofs of other results are not affected.

\subsection*{Acknowledgement} The authors would like to thank 
two anonymous referees for their careful reading of our paper and for 
many very helpful suggestions. Moreover, they thank Giovanni Alessandrini
for pointing out reference \cite{schulz:90} to them. 


\def\cprime{$'$}

\end{document}